%% file: main.tex
\documentclass[reqno,12pt]{amsart}

\input{preamble.tex}

\title[Hopf's conjecture \& positive second intermediate Ricci curvature]{On Hopf's conjecture and positive second intermediate Ricci curvature}

\date{\today}

\author{Lee Kennard}
\address{Syracuse University, Syracuse, New York, USA}
\email{ltkennar@syr.edu}

\author{Lawrence Mouill\'e}
\address{Trinity University, San Antonio, Texas, USA}
\email{lmouille@trinity.edu}

\author{Jan Nienhaus}
\address{University of California, Los Angeles, USA}
\email{nienhaus@math.ucla.edu}


\begin{document}

\begin{abstract}
    Hopf conjectured that even-dimensional closed Riemannian manifolds with positive sectional curvature have positive Euler characteristic. 
    The conclusion of the conjecture is known to fail if the positive sectional curvature assumption is relaxed in any number of ways, including to positive second intermediate Ricci curvature. 
    Here we prove that if a manifold with positive second intermediate Ricci curvature has dimension divisible by four and torus symmetry of rank at least ten, then it has positive Euler characteristic. 
    A crucial new tool is a non-trivial extension of the first author's Four Periodicity Theorem to situations where the periodicity of the cohomology does not extend all the way down to degree zero.
\end{abstract}

\maketitle


In this article, we investigate manifolds with torus symmetry and positive second intermediate Ricci curvature (abbreviated $\Ric_2 > 0$), a condition that is weaker than positive sectional curvature. 
Our first main result is the following:

\begin{main}\label{thm:t8}
    If $\gT^8$ acts on a closed Riemannian manifold $M$ with $\Ric_2 > 0$ and $\dim M \equiv 0 \bmod 4$, then every fixed-point component $F \subseteq M^{\gT^8}$ has $H^{odd}(F;\bQ) = 0$. In particular, $\chi(M) \geq 0$, with equality only if the fixed-point set of $\gT^8$ is empty.
\end{main}

In the presence of positive sectional curvature, any torus action on an even-dimensional manifold has a fixed point. 
Relaxing the curvature condition to $\Ric_2>0$, we generally must pass to a subtorus of codimension two to obtain the same conclusion, as described in the Isotropy Rank Lemma proved by the second author \cite[Proposition E]{Mouille22b}. 
We therefore have the following:

\begin{nonumbercorollary}
    A closed Riemannian manifold with dimension divisible by four, $\Ric_2 > 0$, and $\gT^{10}$ symmetry has positive Euler characteristic.
\end{nonumbercorollary}

In contrast to this result, we have the following examples of even-dimensional manifolds with $\Ric_2>0$ and zero Euler characteristic:
\begin{itemize}
    \item The $6$-manifold $\s^3 \times \s^3$ admits a metric with $\Ric_2 > 0$ and $\gT^3$ symmetry, as established by Wilking (see \cite[Example 2.3]{Mouille22b}). 
    \item The $14$-manifold $\s^7 \times \s^7$ admits a metric with $\Ric_2 > 0$ and $\gT^4$ symmetry, as constructed in \cite{DVDVGARVpreprint}. 
\end{itemize}

To compare our results to the case of positive sectional curvature, one can relax the torus symmetry and remove the divisibility condition on the dimension and still recover positive Euler characteristic. 
The first result of this type where the torus rank does not depend on the manifold dimension is in \cite{KennardWiemelerWilking21preprint} and required $\gT^5$ symmetry.
This result was improved by the third author to show that $\gT^4$ symmetry is sufficient (see \cite{Nienhaus-pre}). 

As for the divisibility assumption, we use it here for the same reason it appeared in earlier work (e.g., \cite{Kennard13}). 
In particular, we apply it to avoid dealing with $(4n+2)$-dimensional manifolds with four-periodic rational cohomology and negative Euler characteristic. 
These are excluded using the $b_3$ Lemma in \cite[Lemma 3.2]{KennardWiemelerWilking21preprint} under special geometric assumptions coming from positive sectional curvature and torus symmetry, but here we would require a stronger form of this result since the Four Periodicity in our proofs is only partial and does not extend all the way to degree zero (see Theorem \ref{thm:P4PT}).

Now generally speaking, an $n$-dimensional Riemannian manifold has \textit{positive $k$-th intermediate Ricci curvature} ($\Ric_k > 0)$ for $k\in\{1,\dots,n-1\}$ if
\[
    \sum_{i=1}^{k} \sec(x, y_i) > 0
\]
for all orthonormal vectors $x,y_1,\dots,y_k$, where $\sec(x,y_i)$ denotes the sectional curvature of the plane spanned by $x$ and $y_i$.
Notice $\Ric_1 > 0$ is equivalent to positive sectional curvature, $\Ric_{n-1}>0$ is equivalent to positive Ricci curvature, and 
$\Ric_k > 0$ implies $\Ric_l > 0$ for all $l \geq k$. 
Thus in the context of intermediate Ricci curvature, $\Ric_2 > 0$ is the condition closest to positive sectional curvature.


Recent work has strengthened {\it constructions} of metrics with positive Ricci curvature to show many of these manifolds also admit positive $\Ric_k$ for some range of values of $k$ less than $n - 1$. One highlight of this work is due to Reiser and Wraith, whose construction shows that Gromov's Betti number bound fails for positive $\Ric_k$ for all values of $k$ somewhat larger than half the manifold dimension \cite{ReiserWraith23}. 
This generalizes the celebrated examples of Sha and Yang \cite{ShaYang91} (cf. \cite{Perelman97,Burdick19,Burdick20,ReiserWraith23preprint}).
For examples of homogeneous spaces with $\Ric_k > 0$ for small values of $k$, see \cite{DVGAM22,AmannQuastZarei22preprint-v2,DVDVGARVpreprint}.

On the other hand, recent work has extended {\it obstructions} for metrics with positive sectional curvature to positive intermediate Ricci curvatures for values of $k$ larger than one. 
For example, there are extensions of the Synge and Weinstein Theorems \cite{Wilhelm97}, the Gromoll-Meyer Theorem and Cheeger-Gromoll Soul Theorem \cite{Shen93,GuijarroWilhelm20}, and the Quarter-pinched Sphere Theorem \cite{Shen90,Xia97,GuijarroWilhelm22}.
In \cite{Mouille22b}, the second author extended the Berger-Sugahara Isotropy Rank Lemma, Grove-Searle Symmetry Rank Bound, and Wilking Connectedness Lemma with symmetries (building on work by Guijarro and Wilhelm from \cite{GuijarroWilhelm22} in the latter case; see Theorem \ref{thm:CL} below). 
The first and second authors in \cite{Mouille22b,KennardMouille24} subsequently extended the the Grove-Searle Maximal Symmetry Rank Theorem \cite{GroveSearle94} in large dimensions.

We note that for any $k \geq 1$, there exists a natural number $n(k)$ such that in all dimensions larger than $n(k)$, the only known examples of closed simply connected manifolds with $\Ric_k > 0$ are the rank one symmetric spaces. 
Therefore, motivated by our first main result and \cite{KennardWiemelerWilking21preprint,Nienhaus-pre}, one might ask whether there exists $d(k)$ such that any closed even-dimensional manifold $M$ with $\Ric_k > 0$ and $\gT^{d(k)}$ symmetry has positive Euler characteristic. 
However, a key ingredient for the $k=2$ case considered here is the fact that the $\Ric>0$ implies $b_1(M)=0$ by Myers theorem, which in conjunction with the Connectedness Lemma and Periodicity Lemma as described below, allows us to show that the odd Betti numbers vanish. 
Because these tools are weaker under the assumption $k\geq 3$, new ideas would be required in this setting.



Another major conjecture for manifolds with lower bounds on sectional curvature is the Ellipticity Conjecture due to Bott, Grove, and Halperin (see \cite{Grove02, GroveHalperin82}). The conjecture states that a manifold with almost non-negative sectional curvature is rationally elliptic. In even dimensions, this condition implies that the Euler characteristic is positive if and only if the Betti numbers vanish in odd degrees. Theorem \ref{thm:t8} may be viewed as evidence for this since it provides a confirmation for manifolds arising as fixed-point components of $\gT^8$-actions on larger, positively curved manifolds. 

Under additional geometric assumptions, we can also prove the vanishing of odd Betti numbers for the ambient manifold. The main technical result for achieving this is the following partial analogue of Theorem A in \cite{KennardWiemelerWilking25}, which is the same except for a stronger positive curvature assumption and correspondingly a stronger cohomological conclusion.

\begin{main}\label{thm:tcig}
If $\gT^9$ acts with connected isotropy groups and a non-trivial fixed-point set on a closed, orientable Riemannian manifold $M$ with $\Ric_2 > 0$, then the rational cohomology of $M$ is four-periodic from degree $1$ to degree $n - 1$. 
\end{main}

The conclusion means that there exists $x \in H^4(M; \bQ)$ such that multiplication by $x$ induces maps $H^i(M; \bQ) \to H^{i+4}(M; \bQ)$ that are surjective for $i = 1$, injective for $i = n - 5$, and bijective for all $1 < i < n - 5$. This notion of periodicity was first discovered to arise in positive curvature by Wilking \cite{Wilking03}. 
An analysis of this condition showing that in special situations this implies Four Periodicity in all degrees was carried out by the first author \cite{Kennard13}, and an extension of this work to partial periodicity as stated in Theorem \ref*{thm:tcig} was first systematically analyzed by the third author in his M.Sc. thesis \cite{NienhausMSc}.

To ensure the existence of a $\gT^9$-action with a fixed-point under the condition $\Ric_2>0$, we may pass to a codimension two subtorus of a $\gT^{11}$ that acts on $M$. Moreover in dimensions divisible by four, the existence of a $\gT^{10}$-action implies positive Euler characteristic by Theorem \ref{thm:t8}, so the entire $\gT^{10}$ already has a fixed-point. Combined with Poincar\'e duality in dimensions congruent to zero or one modulo four, we obtain two especially nice consequences:

\begin{nonumbercorollary}
If $\gT^{10}$ acts with connected isotropy groups on a closed Riemannian manifold $M^{4n}$ with $\Ric_2 > 0$, then $H^{odd}(M; \bQ) = 0$.
\end{nonumbercorollary}

\begin{nonumbercorollary}
If $\gT^{11}$ acts with connected isotropy groups on a closed Riemannian manifold $M^{4n+1}$ with $\Ric_2 > 0$, then 
$M$ is a rational homology sphere.
\end{nonumbercorollary}

To compare these corollaries to what is known for manifolds with positive sectional curvature, one can relax the assumptions in the first corollary to $\gT^9$ symmetry, drop the condition that $\dim M \equiv 0 \bmod 4$, and strengthen the conclusion to state that the rational cohomology agrees with a sphere, complex projective space, or quaternionic projective space in all degrees in the even-dimensional case. 
Similarly in the second corollary, one can relax to $\gT^{10}$ symmetry, drop the condition $\dim M \equiv 1 \bmod 4$, and derive the same conclusion in odd dimensions. Since there are no additional known examples admitting $\Ric_2 > 0$ and torus symmetry in this range beyond those admitting positive sectional curvature, further exploration is required to determine whether Theorem \ref{thm:tcig} and its corollaries are optimal.

We already discussed extensions of the Isotropy Rank Lemma and Wilking's Connectedness Lemma to the case of positive intermediate Ricci curvature. 
These continue to be crucial tools in our work. 
Another major tool is the $\gS^1$ Splitting Theorem of the first author, Wiemeler, and Wilking \cite[Theorem F]{KennardWiemelerWilking21preprint}. As this result does not involve curvature, it carries over to the setting of $\Ric_2 > 0$.

A key tool from positive sectional curvature that does not easily extend to $\Ric_2 > 0$ is the Four Periodicity Theorem of the first author \cite[Theorem C]{Kennard13}. 
The main technicality is that the periodic cohomology arising from Wilking's Connectedness and Periodicity Lemmas does not start at degree zero. Indeed, the weaker positive curvature  condition leads to a weaker conclusion in the Connectedness Lemma. The main problem then is that the cohomology group $H^k(M;\Z_p)$ of a $k$-periodic cohomology ring is not necessarily one-dimensional. This assumption is crucial to the first author's proof of the Four Periodicity Theorem.

This technicality is difficult to overcome, but the third author managed to do so in his 2018 M.Sc. thesis. There are two key ideas. The first is the {\it definition of irreducibly periodic cohomology} and an observation that these algebras behave enough like (fully) periodic cohomology so that the ideas in \cite{Kennard13} can be made to work. The second is a non-trivial {\it Decomposition Lemma}, which states that partially periodic cohomology algebras decompose as a sum of irreducible periodic cohomology algebras. 
As an example, a connected sum of complex projective spaces, in a sense made precise below, decomposes into the sum of two copies of two-periodic cohomology rings. 
The Decomposition Lemma then allows the proof of the Partial Four Periodicity Theorem to be reduced to the irreducible case.

As these results on partially periodic cohomology have neither appeared nor been applied in published form, we include their proofs here. An illustrative consequence is the following (see Section \ref{sec:Partial4PT}):

%
%
\begin{main}\label{thm:bodd-intro}
If $N^{n-k} \to M^n$ is a $(n - k - 1)$-connected inclusion of closed manifolds with finite fundamental group and $3k\leq n - 1$, then the universal cover of $M$ has $\gcd(4,k)$-periodic rational cohomology on degrees $1$ to $n-1$. 

In particular, the odd Betti numbers of $M$ vanish when $k \equiv 2 \bmod 4$ or $n \equiv 0 \bmod 4$ in even dimensions, and $M$ is a rational sphere if $n \equiv 1 \bmod 4$. 
\end{main}

We remark that fundamental groups are required to be finite given $\Ric_2 > 0$ and dimension larger than two by Myers Theorem and that an $(n-k-1)$-connected inclusion as in Theorem \ref{thm:bodd-intro} arises in this context when two totally geodesic submanifolds intersect transversely in a manifold with $\Ric_2 > 0$. 
We postpone the definition of Four Periodicity until Section \ref{sec:Subquotient}. 

We also remark that {\it all even-dimensional manifolds known to admit positive sectional curvature} have $2$-, $4$-, or $8$-periodic rational cohomology away from the bottom and top degrees (see \cite{Eschenburg92Manuscripta}, \cite[Remark 4.7]{GoertschesKonstantisZoller20GKM},\cite{MareWillems13}, and \cite{AmannZiller16}), and this {\it periodicity is irreducible} in the sense mentioned above.\footnote{The 2-periodicity of the 6-dimensional space of Eschenburg is irreducible over the rationals, but reduces over the reals. All other known examples are irreducible in all coefficients.} 
In particular, the Wallach flag manifolds satisfy this property. 
By contrast, examples such as the Cheeger manifolds $\CP^n \# \CP^n$ and $\HP^n \# \HP^n$, which admit metrics of non-negative sectional curvature but are not expected to admit positive sectional curvature, respectively have $2$- and $4$-periodic cohomology away from the endpoints, but not irreducibly so. 

{\bf Organization}. 
Sections \ref{sec:Subquotient} - \ref{sec:Partial4PT} contain the proof of the extension of the Four Periodicity Theorem to the case of cohomology rings that are only partially periodic. 
Section \ref{sec:t8} contains the proof of Theorem \ref{thm:t8}, and Section \ref{sec:t9} contains the proof of Theorem \ref{thm:tcig}.

{\bf Acknowledgements}. The first author was partially supported by NSF Grant DMS-2402129, the Simons Foundation TSM program, and NSF Grant DMS-1928930 while in residence at SLMath in Berkeley, California, in Fall 2024. The second author was supported by NSF Award DMS-2202826 at Syracuse University and a summer stipend at Trinity University. The third author proved the main results of Sections \ref{sec:Subquotient} - \ref{sec:Partial4PT} in his M.Sc. thesis at WWU M\"unster and would like to thank his advisor Burkhard Wilking.

\section{The periodic subquotient}\label{sec:Subquotient}

In this section, we define a procedure for turning a partially periodic cohomology ring into a fully periodic subquotient of that cohomology ring. The additive, multiplicative, and Steenrod module structure are essentially the same, and working with this subquotient proves to be convenient in our proofs.

\begin{definition}[cf. Definition 2.3 in \cite{Nienhaus-pre}] 
For a commutative ring $\Lambda$, we say $x \in H^k(M; \Lambda)$ induces periodicity in degrees $1 \leq * \leq n - 1$ if $3k \leq n - 1$ and the maps $H^i(M; \Lambda) \to H^{i+k}(M; \Lambda)$ given by cup products with $x$ are surjective for $1 \leq i < n - 1 - k$ and injective for $1 < i \leq n - 1 - k$. We also say that $x \in H^k(M; \Lambda)$ induces periodicity on $H^*(M; \Lambda)$ if it has degree at most $n - 1$ and is a product of such elements. In either case, we say $H^*(M; \Lambda)$ is $k$-periodic on degrees $1 \leq * \leq n - 1$.
\end{definition}


We remark on this definition compared to previous notions. First, the phrasing is closest to the that in \cite{Nienhaus-pre}, but the condition $2k \leq n$ has been replaced by $3k \leq n - 1$. To explain why this is reasonable, we note that periodicity from degree $0$ to $n$ required the assumption $2k \leq n$ so that all cohomology groups participate in the periodicity. In our setting, where the periodicity goes from degree $1$ to $n - 1$, the analogous condition would be $2k \leq n - 2$. However, we will see that the stronger condition $2k \leq n - 1 - k$, or equivalently $3k \leq n - 1$, is required to allow arguments from \cite{Kennard13,Nienhaus-pre} to carry over. 

We also remark that when $p$ is an odd prime, that $k$-periodicity of $H^*(M;\bZ_p)$ implies that $k$ is even or $M$ has vanishing $\bZ_p$-cohomology in degrees $2, \ldots, n-2$ since the square of an element $x \in H^k(M; \Z_p)$ is zero if $k$ is odd, which implies $x=0$ since $\cup x$ is injective in degree $k>1$. If $k=1$, one can also see that the claimed cohomology groups must vanish even if $x$ does not.

\begin{example}
If $N^{n-k} \to M^n$ is an $(n - k - 1)$-connected inclusion of closed, orientable manifolds with $3k \leq n - 1$, then the Periodicity Lemma from \cite{Wilking03} (see Theorem \ref{thm:PT} below) implies the existence of $x \in H^k(M)$ inducing $k$-periodicity on degrees $1 \leq * \leq n - 1$. 
\end{example}

To clean up activity at the endpoints, we fix an element $x \in H^k(M; \Z_p)$ inducing periodicity with $3k \leq n - 1$ and define the graded $\bZ_p$-vector space $\bar H^*$ with indices $1 \leq * \leq n - 1$ by the formula
	\[\bar H^i = \left\{
			\begin{array}{rcl}
			H^1(M;\bZ_p)/\ker(x)	&\mathrm{if}&	i = 1\\
				H^i(M; \bZ_p)	&\mathrm{if}&	1 < i < n - 1\\
				\im(x)			&\mathrm{if}&	i = n - 1
			\end{array}
			\right.\]
where $\ker(x)$ and $\im(x)$ denote the kernel and image, respectively, 
of the maps given by multiplication by $x$ in the appropriate degrees. 
We refer to $\bar H^*$ as a subquotient of $H^*(M; \bZ_p)$ since it is a quotient of a subspace of $H^*(M; \bZ_p)$. We remark that one could alternatively identify $\bar H^1$ with a choice of complement of $\ker(x) \subseteq H^1(M; \bZ_p)$ since we are working with coefficients in a field and therefore consider $\bar H^*$ as a subspace. It is immediate to see that multiplication by $x$ induces isomorphisms
	\begin{equation}\label{eqn:xi}
	    \xi:\bar H^i \to \bar H^{i+k}
        \end{equation}
for all $1 \leq i \leq n - k - 1$. 
Thus, we have inverse maps
	\[\xi^{-1}:\bar H^j \to \bar H^{j-k}\]
for all $1+k \leq j \leq n - 1$ and, in particular, for all $k + 1 \leq j \leq 2k + 1$.

\begin{proposition}
If $H^*(M; \bZ_p)$ is $k$-periodic, then its ring structure descends to $\bar H^*$. 
Similarly, if $pk \leq n - 1$, then $\bar H^*$ inherits a module structure over the Steenrod algebra.
\end{proposition}

\begin{proof}
By definition, there exists $x \in H^k(M; \bZ_p)$ inducing periodicity with $3k \leq n - 1$. For the proof, we extend the definition of $\bar H^*$ to all degrees $0 \leq * \leq n$ by setting $\bar H^0$ to be the zero-dimensional subspace of $H^0(M; \Z_p)$ and $\bar H^n$ to be the zero-dimensional quotient of $H^n(M; \Z_p)$.

Addition in $H^*(M; \Z_p)$ clearly induces a unique additive structure on $\bar H^*$. For example, if $y, z \in \bar H^{n-1}$, then $y = x y'$ and $z = x z'$ for some $y', z' \in \bar H^{n-1-k}$, and hence $y + z$ equals $x (y' + z')$ and is in $\bar H^{n-1}$.

To see that multiplication descends to the subquotient, fix $a \in \bar H^i$ and $b \in \bar H^j$ with $i \leq j$. If $i, j, i+j \not\in\{0,1,n-1,n\}$, the product $ab \in \bar H^{i+j}$ is clearly well defined. If $i = 0$, then $a = 0$ and hence $ab = 0 \in \bar H^j$. If $i = 1$, we need to show that $ax = 0$ implies $ab = 0$. For $j \leq n - 1 - k$, this holds because $(ab)x = 0$ and multiplication by $x$ is injective into degree $i + j + k$. For $n - k \leq n - 1$, this holds because $b$ factors as $xb'$ and hence $ab = (ax)b' = 0$. If $i, j > 1$ and $i + j = n - 1$, then $j \geq k + 1$ since $2k \leq n - 2$ and $i \leq j$, so $a b = a(xb') \in \im(x) = \bar H^{n-1}$. Finally if $i, j > 1$ and $i + j = n$, then $a b$ may be non-zero in $H^n(M;\bZ_p)$ but then in the quotient $\bar H^n$ it is zero and hence is well defined.

It suffices to prove that $\bar H^*$ inherits an action by the Steenrod algebra when $pk \leq n - 1$. First assume $p = 2$. We claim for any $i \geq 0$ that
    \[\Sq^i:H^j(M; \Z_p) \to H^{j+i}(M; \Z_p)\]
induces a well defined homomorphism $\Sq^i:\bar H^j \to \bar H^{j+i}$. The claim holds 
for $i = 0$ since $\Sq^0$ is the identity, 
for $i = j = 1$ since $a \in H^1(M; \bZ_p)$ such that $ax = 0$ implies $\Sq^1(a)x = a^2x = 0$ and hence that $a^2 = 0$ by periodicity, and 
for $i > j$ since $\Sq^i = 0$ on $H^j(M; \Z_p)$ in this case. 
We may therefore assume $1 \leq i \leq j$ and $j \neq 1$. 
If $i + j < n - 1$, then there is nothing to show since $\bar H^j = H^j(M; \bZ_p)$ and $\bar H^{i+j} = H^{i+j}(M; \Z_p)$. Similarly if $i + j = n$, then $\bar H^{i+j} = 0$ and there is nothing to show. If $i + j = n - 1$, then $j \geq \tfrac{n-1}{2} > k$, so any $u \in \bar H^j$ factors as $x u'$. By the Cartan formula, we have
	\[\Sq^i(u) = \Sq^0(x) \Sq^i(u') + \sum_{1 \leq h \leq i} \Sq^h(x) \Sq^{i-h}(u').\]
For $1 \leq h \leq i$, we have $1 +k \leq |\Sq^h(x)| < n - 1$, so $\Sq^h x$ factors as $x x_h$ for some $x_h \in \bar H^h$. Since also $\Sq^0 x = x$, the right-hand side is in $\im(x) = \bar H^{n-1}$, as needed.

Finally, we claim that $\bar H^*$ inherits an action by the Steenrod algebra when $p \geq 3$ and $pk \leq n -1$. 
We fix $i \geq 0$ and claim that $\P^i:H^j(M;\bZ_p) \to H^{j + 2(p-1)i}(M;\bZ_p)$ induces a well defined homomorphism $\P^i:\bar H^j \to \bar H^{j + 2(p-1)i}$ for all $j$. 
As in the case where $p = 2$, the claim holds trivially if $i = 0$ or $2i > j$, as the map $\P^i$ is either the identity or the zero map. 
Similarly, the claim holds trivially if $2(p-1)i + j \neq n - 1$, since the left-hand side of this inequality is the degree of $\P^i(u)$. 
It suffices to consider the case where $i \geq 1$, $2i \leq j$, and $2(p-1)i + j = n - 1$. 
If moreover $j > k$, then any $u \in \bar H^j$ factors as $x u'$ and the claim again follows by an argument using the Cartan formula. 
If instead $j \leq k$, then the assumption $pk \leq n - 1$ implies that $2i = j = k$. 
In particular, $\P^i(u) = u^p = u^2 u^{p-2}$. 
Since $|u| = k$ and $p \geq 3$, we may factor $u^2$ as $x u'$ and conclude that $\P^i(u) \in \im(x) = \bar H^{n-1}$, as required.
\end{proof}

\section{Irreducible periodicity}\label{sec:IrreduciblePeriodicity}

In this section, we define a special case of periodicity called irreducible periodicity, and we prove that many properties of (fully) periodic cohomology rings also hold in the case of irreducible (partial) periodicity. In particular, we prove versions of the $\Z_p$-Periodicity Theorems in \cite{Kennard13} in the case of irreducible partial periodicity.

\begin{definition}
The periodic subquotient $\bar H^*$ is irreducibly periodic via $x$ if $x$ induces periodicity and has the property that $x = a + b$ only if $a$ or $b$ also induces periodicity.
\end{definition}

\begin{example}
For a connected $n$-manifold $M^n$, if $2k \leq n$ and there exists a non-zero $x \in H^k(M)$ inducing periodicity from degree $0$ to $n$ in the sense of \cite{Kennard13,Nienhaus-pre}, then $H^k(M;\Z_p)$ is one-dimensional because $H^0(M;\bZ_p)$ surjects onto it. In particular, if $x$ is a sum of the form $a + b$, then at least one of those summands is a non-zero multiple of $x$ and hence induces periodicity. In that sense, all periodic cohomology rings analyzed in \cite{Kennard13} were automatically irreducibly periodic. 
\end{example}

Another way to phrase the definition of irreducibility is that 
	\[N^k = \{y \in \bar H^k \mid y \mathrm{~does~not~induce~periodicity}\}\]
is a vector subspace of $\bar H^k$. We note that inhomogeneous elements in $\bar H^*$ do not induce periodicity, so more generally the subset $N$ of $\bar H^*$ of elements that do not induce periodicity is a vector subspace of $\bar H^*$. 

The next two results we need are that moreover $N$ is an ideal (Lemma \ref{lem:Products+Factors} below) and a submodule under the action of the Steenrod algebra in special circumstances (see Lemma \ref{lem:partialP^i-module} below). We note however that the proof of Lemma \ref{lem:Products+Factors} does not require irreducibility.

\begin{lemma}[Inheritance by products and factors]\label{lem:Products+Factors}
If $H^*$ is irreducibly periodic and $y, z \in \bar H^*$ are homogeneous elements with $|yz| \leq n-1$, then $yz$ induces periodicity if and only if $y$ and $z$ do. As a consequence, the minimal period divides any other period.
\end{lemma}

The proof is based on the fact that multiplication by $yz$ factors as multiplication by $y$ followed by multiplication by $z$ and that the latter two maps commute up to sign. However the proof is more involved than for periodicity from $0$ to $n$, so we include in for completeness.

\begin{proof}
If $y$ and $z$ induce periodicity, then each is equal to or a product of elements of degree at most $\tfrac{n-1}{3}$ that induce periodicity. In either case, $yz$ is clearly a product of such elements and hence induces periodicity by definition.

For the converse, suppose $yz$ induces periodicity. Recall there are two possibilities in the definition of periodicity. For this proof, we break these into Cases 1, 2, and 3 below.

{\bf Case 1}. $yz$ has degree at most $\tfrac{n-1}{3}$. 

This works as in \cite{Kennard13} and \cite{Nienhaus-pre}, but we give the details here. Consider the composition of the maps given by multiplication by $y$, $z$, and then $y$ again. 
The compositions of the first two and of the last two are isomorphisms as long as their domains and codomains remain within degrees from $1$ to $n-1$. 
Thus multiplication by $y$ induces maps $\bar H^i \to \bar H^{i+|y|}$ that are isomorphisms for $1 \leq i \leq n - 1 - |yz| - |y|$ and for $|yz|+1 \leq i \leq n - 1 - |y|$. 
By the assumption in Case 1, these degree ranges cover all possible degrees, so $y$ induces periodicity. 
A symmetric argument works for $z$, so the proof is complete in Case 1.

{\bf Case 2}: $yz$ is a non-trivial product of elements of degree at most $\tfrac{n-1}{3}$ that induce periodicity, and moreover one of those elements has degree at most $\tfrac{n-2}{3}$.

By isolating one factor of minimal degree and grouping the others into one term, we may assume $yz = uv$ where $u$ and $v$ induce periodicity, $0 < |u| \leq |v|$, and $|u| \leq \tfrac{n-2}{3}$.

If $|y| > |u|$, we can use periodicity to write $y = uy'$ and obtain the equation $y' z = v$. By induction over $|yz|$ and the fact that $v$ induces periodicity, we conclude that $y'$ and $z$ induce periodicity. By the first part of the lemma, $y$ also induces periodicity.

If $|y| \leq |u|$ and moreover $|z| \leq |u|$ without loss of generality, we first look at the composition of multiplication by $y$ followed by multiplication by $z$. Since $yz$ induces periodicity, we see that multiplication by $y$ induces injections $\bar H^i \to \bar H^{i+|y|}$ for $1 \leq i \leq (n-1) - |y| - |z|$. Working inductively over $i$, we then look at the two possible compositions of multiplication by $y$ and multiplication by $u$, which is always an isomorphism, to conclude that multiplication by $y$ induces injections $\bar H^i \to \bar H^{i + |y|}$ for all $1+|u| < i \leq (n-1) - |y|$. Since $|y|$, $|z|$, and $|u|$ are at most $\tfrac{n-2}{3}$, we see that multiplication by $y$ induces injections in all required degrees. A similar proof shows surjectivity, so the proof is complete.

{\bf Case 3}: $yz$ is a non-trivial product of elements of degree $\tfrac{n-1}{3}$ that induce periodicity.

As in Case 2, we may assume by induction that $|y| \leq \tfrac{n-1}{3}$ and $|z| \leq \tfrac{n-1}{3}$. Hence we may assume $yz = uv$ where $u$ and $v$ induce periodicity and
    \[|y| = |z| = |u| = |v| = \tfrac{n-1}{3}.\]
If $y$ or $z$ induces periodicity, then an argument similar to that in Case 1 and using the bounds $|y|, |z| \leq \tfrac{n-1}{3}$ shows that both $y$ and $z$ induce periodicity as well. 

If neither $y$ nor $z$ induces periodicity, we derive a contradiction as follows. By irreducibility, the equations $u = y + (u-y)$ and $v = -z + (v+z)$ imply that both $u - y$ and $v + z$ induce periodicity. By the first part of the lemma, the product $(u - y) (v + z)$ induces periodicity. This product equals $y v - z u$, so by irreducibility again $y v$ or $z u$ induces periodicity. By the argument in the last paragraph, we see that $y$ or $z$ induces periodicity, a contradiction.
\end{proof}

\begin{remark}
    As the proof shows, Lemma \ref{lem:Products+Factors} remains true without assumptions of irreducibility if the minimal period is assumed to satisfy $3k\le n - 2$ instead of the condition $3k \leq n - 1$ stated in the definition of periodicity. We prove the stronger version to simplify some proofs later on.
\end{remark}

Before proving the submodule property, we note the following useful consequence of Lemma \ref{lem:Products+Factors}.

\begin{corollary}\label{lem:AllPeriodsIrreducible}
If $\bar H^*$ is irreducibly periodic, then any element inducing periodicity does so irreducibly.
\end{corollary}

\begin{proof}
Assume $x \in H^k(M)$ induces periodicity and has the property that $x = x_1 + x_2$ only if $x_1$ or $x_2$ also induces periodicity. 

If the lemma does not hold, then there exists $y \in H^l(M)$ that induces periodicity and there exists a decomposition $y = y_1 + y_2$ such that neither $y_1$ nor $y_2$ induces periodicity. Assume $y$ is such an element of minimal degree.

First, if $|y| \geq k + 1$, then periodicity via $x$ implies a factorization $y = x z$. Since also $|y_i| \geq k + 1$, there exist factorizations $y_i = x z_i$. Periodicity via $x$ then implies that $z = z_1 + z_2$. By Lemma \ref{lem:Products+Factors}, we have that $z$ induces periodicity but that neither $z_i$ induces periodicity. Hence we have a contradiction to the minimality of $|y|$.

We may assume $|y| \leq k$. Periodicity via $y$ implies that $x^2 = y z$ for some $z \in \bar H^*$. The decomposition of $y$ implies $x^2 = y_1 z + y_2 z$. Now periodicity via $x$ implies $y_i z = x w_i$ and that $x = w_1 + w_2$. Since $x$ induces periodicity irreducibly, some $w_i$ induces periodicity. Lemma \ref{lem:Products+Factors} then implies that $y_i z$ induces periodicity for some $i$ and hence that $y_i$ induces periodicity for some $i$, a contradiction to our choice of $y$.
\end{proof}

Before proving that the ideal $N$ is also a submodule under the action of the Steenrod algebra, we recall the $\iota$ functional introduced in \cite{NienhausMSc}.

\begin{definition}[$\iota$ functional]
For $y \in \bar H^*$ such that $P^i(y)$ induces periodicity for some $i \geq 0$, the integer $\iota(y)$ is defined to be the minimum such $i$.
\end{definition}

The following result is similar to \cite[Lemma 2.11]{Nienhaus-pre} and is a consequence of the Cartan formula, Lemma \ref{lem:Products+Factors}, and the property of $k$-periodicity that $\dim(\bar H^k) = 1$. The last property is the only one we do not have here, but the condition of {\it irreducible} periodicity serves as a substitute. The proof therefore extends to any irreducibly $k$-periodic cohomology ring, and we include the proof for completeness.

\begin{lemma}[Additivity of $\iota$]\label{lem:Additivity}
Assume $\bar H^*$ is irreducibly $k$-periodic, and assume $x, y, z \in \bar H^*$ are homogeneous elements such that $x = y z$. If
	\begin{enumerate}
	\item $\iota(x)$ is defined and $P^{\iota(x)}(x)$ has degree at most $n-1$, or 
	\item $\iota(y)$ and $\iota(z)$ are defined and $P^{\iota(y)}(y) P^{\iota(z)}(z)$ has degree at most $n-1$,
	\end{enumerate}
then all three are defined and $\iota(x) = \iota(y) + \iota(z)$. 
\end{lemma}

We note that the equation $P^{\iota(x)}(x) = P^{\iota(y)}(y) P^{\iota(z)}(z)$ in \cite[Lemma 2.10]{Nienhaus-pre} does not extend to our setting. While we do not need this fact here, we note that the proof below shows that this equality does hold modulo the ideal $N$.

\begin{proof}
    If $\iota(x)$ is defined and $x = yz$, then $P^{\iota(x)}(yz)$ induces periodicity and moreover irreducibly so by Lemma \ref{lem:AllPeriodsIrreducible}. At the same time, we have
    \[
        P^{\iota(x)}(yz) = \sum_{i + j = \iota(x)} P^i(y) P^j(z)
    \]
    by the Cartan formula, so $P^i(y) P^j(z)$ induces periodicity for some $i + j = \iota(x)$. By the Products and Factors Lemma (Lemma \ref{lem:Products+Factors}), so does $P^i(y)$ and $P^j(z)$, so by definition of the $\iota$ functional, we have $\iota(y) + \iota(z) \leq \iota(x)$. To prove equality, we could run the same argument with $\iota(x)$ replaced by $\iota(y) + \iota(z)$. Since the term $P^{\iota(y)}(y) P^{\iota(z)}(z)$ induces periodicity, by irreducible periodicity so does either $P^{\iota(y) + \iota(z)}(x)$ or some other term of the form $P^i(y) P^j(z)$ for some $i + j = \iota(y) + \iota(z)$ with $(i, j) \neq (\iota(y), \iota(z))$. In either case, we have a contradiction to either the definition of $\iota(x)$ or the Products and Factors Lemma.

    If $\iota(y)$ and $\iota(z)$ are defined, then we look again at $P^{\iota(y) + \iota(z)}(yz)$ and argue similarly to conclude that $\iota(x) \leq \iota(x) + \iota(y)$. Finally equality holds by again considering the expression for $P^{\iota(x)}(yz)$.
\end{proof}

With the $\iota$ functional and these properties established, we proceed to the establish the submodule property. Phrased in a positive way about inducing periodicity, the statement is the following:

\begin{lemma}[Inheritance by preimages under Steenrod powers]\label{lem:partialP^i-module} 
Assume $\bar H^*$ is irreducibly $k$-periodic and $y \in \bar H^*$ is a homogeneous element. Fix $i \geq 0$.
	\begin{enumerate}
	\item If $p = 2$ and $\Sq^i(y)$ induces periodicity, then $y$ also induces periodicity. 
	\item If $p \geq 3$ and $\P^i(y)$ induces periodicity, then $y$ induces periodicity as well if $2(p-1)k \leq n-1$ or if $pk \leq n - 1$ and $p |P^{\iota(y)}(y)| \leq n - 1$.
	\end{enumerate}
\end{lemma}

The definition of irreducibility and Lemmas \ref{lem:Products+Factors} and \ref{lem:partialP^i-module} imply that 
	\[N = \{y \in \bar H^* \mid y \mathrm{~does~not~induce~periodicity}\}\]
is both an ideal of $\bar H^*$ and a submodule of $\bar H^*$ with respect to the action of the Steenrod algebra when $p = 2$ or when $p \geq 3$ and $2(p-1)k \leq n - 1$. While we do not use this fact explicitly, we make this remark since it might help remember these properties.

\begin{remark}
In the case where $H^*(M; \Z_2)$ is $k$-periodic via a non-zero element $x$ in all degrees from $0$ to $n$ and $2k \leq n$, then this discussion applies noting $\bar H^* = H^*(M; \Z_2)$ in this case and that the definition of periodicity in this case has the bound $k\leq \tfrac{n-1}{3}$ replaced by $k \leq \tfrac{n}{2}$. If $k$ is minimal, then $N^i = H^i(M;\Z_2)$ for all $i \neq 0 \bmod k$, and the quotient space $\bar H^*/N$ is isomorphic to a singly generated polynomial algebra $\Z_2[x]/(x^{\floor{n/k} + 1})$. Since $\bar H^*/N$ is also a module over the Steenrod algebra, the result of Adem \cite{Adem52} implies that $k$ is a power of two. This provides repackaged proof of the main $\Z_2$-periodicity theorem in \cite{Kennard13}. If moreover $\bar H^*/N$ could be realized as a the cohomology of a topological space $\bar M$, then Adams's improvement of Adem's result implies that either $k \in \{1, 2, 4\}$ or $k = 8$ and $x^3 = 0$ (see \cite{Adams60}). This would prove the Four Periodicity Conjecture in \cite[Section 6]{Kennard13} (see also \cite{HarperKennard-pre}).
\end{remark}

The proof of Lemma \ref{lem:partialP^i-module} is a simplification of arguments in both \cite{Kennard13} and \cite{Nienhaus-pre}. 

\begin{proof}[Proof of Lemma \ref{lem:partialP^i-module}] 
The proofs for $p = 2$ and $p \geq 3$ are similar, so we include here only the proof for $p \geq 3$. In particular, we may assume $pk \leq n - 1$. Without loss of generality, we may assume that $k$ is the minimum period. In particular, $k$ divides any other period by Lemma \ref{lem:Products+Factors}.

Let $y \in \bar H^*$ be a minimal counterexample to the claim in the following sense: $y$ does not induce periodicity, $P^i(y)$ induces periodicity for some $i > 0$, $|P^{\iota(y)}(y)| \leq \tfrac{n-1}{p}$ if $2(p-1)k > n - 1$, and $|y|$ is minimal among all such counterexamples. Note by minimality that $|y| \leq k$ since otherwise we can factor $y$ by periodicity to obtain an element $y'$ with $|y'| < |y|$ such that $y'$ does not induce periodicity (by Lemma \ref{lem:Products+Factors}) but $P^{i'}(y')$ does for some $0 < i' \leq i$ (by Lemma \ref{lem:Additivity}).

The proof considers an auxiliary element $z \in H^*$ (possibly equal to $y$) with the following properties: $z$ does not induce periodicity,  $P^j(z)$ does induce periodicity for some $j > 0$, $\iota(z)$ is maximal among all such choices for $z$, and moreover $|z|$ is minimal among such maximizers. As with the element $y$, we have that $|z| \leq k$ by minimality.

If zero induces periodicity, $\bar H^* = 0$ and the lemma holds trivially. Otherwise, the fact that $P^{\iota(z)}(z)$ induces periodicity implies it is non-zero. In particular, $|P^{\iota(z)}| \leq |z^p|$. Moreover, $P^{\iota(z)}(z) \neq z^p$ since otherwise $z$ would induce periodicity by Lemma \ref{lem:Products+Factors}. 
Therefore $|P^{\iota(z)}(z)|$ is both strictly less than $|z^p| = p |z| \leq pk$ and divisible by $k$, so we have that $|P^{\iota(z)}(z)| \leq (p-1)k$. 

The same argument shows that $|P^{\iota(y)}(y)| \leq (p-1)k$, and we recall that moreover $|P^{\iota(y)}(y)| \leq \tfrac{n-1}{p}$ if $2(p-1) k > n - 1$. Combined with the assumptions of the lemma, 
we see that the product $P^{\iota(y)}(y) P^{\iota(z)}(z)$ has degree 
	\[|P^{\iota(y)}(y)| + |P^{\iota(z)}(z)| \leq n - 1.\]
In this range, additivity of the $\iota$ functional holds, and we have $\iota(yz) = \iota(y) + \iota(z)$. Since $\iota(y) > 0$, we have a contradiction to the maximality of $\iota(z)$.
\end{proof}

We are ready to prove an extension of the $\Z_p$-periodicity theorem from \cite{Kennard13} to the case of irreducible partial periodicity.

\begin{theorem}[Partial $\Z_2$-periodicity theorem, irreducible case]
If $H^*(M; \Z_2)$ is irreducibly $k$-periodic on degrees $1 \leq * \leq n - 1$, then the minimum period is of the form $2^\alpha$ with $\alpha \geq 0$.
\end{theorem}

\begin{proof}
Let $\bar H^*$ be the associated periodic subquotient. Since every element inducing periodicity does so irreducibly by Lemma \ref{lem:AllPeriodsIrreducible}, we may assume $k$ is the minimum period. We assume that $k$ is not a power of two and derive a contradiction.

The Adem relations imply a relation of the form
	\[\Sq^k = \sum \Sq^{2^i} \circ Q_i\]
for some Steenrod algebra elements $Q_i$ where the sum runs over $i \geq 0$ such that $2^i < k$. Evaluating on $x$ yields an expression of the form
	\[x^2 = \sum \Sq^{2^i}(y_i)\]
for some $y_i \in \bar H^*$. Since $x$ induces periodicity, so does $x^2$ by Lemma \ref{lem:Products+Factors}. Moreover, this periodicity is irreducible, so some element of the form $\Sq^{2^i}(y_i)$ induces periodicity as well.

First suppose that $|y_i| \leq k$. The element $\Sq^{2^i}(y_i)$ has degree less than $2k$, which is at most $n - 2$, so Lemma \ref{lem:partialP^i-module} implies that $y_i$ induces periodicity. But then $|y_i|$ is divisible by $k$, $|y_i| = k$, and $2^i = 2k - |y_i| = k$, so we have a contradiction to the fact that $2^i < k$.

Suppose now that $|y_i| > k$. Using periodicity, we may write $y_i = x z$. By the additivity of the $\iota$ functional, we have $\iota(z) = \iota(y_i) \leq 2^i$, so we have an element of the form $\Sq^j(z)$ that induces periodicity with the property that $j < k$ and $|z| \leq k$. As in the previous paragraph, this leads to a contradiction.
\end{proof}

\begin{theorem}[Partial $\Z_p$-periodicity theorem, irreducible case] 
Let $p$ be an odd prime. If $H^*(M; \Z_p)$ is irreducibly $k$-periodic on degrees $1 \leq * \leq n - 1$ with $pk \leq n - 1$, then 
the minimum such $k$ satisfies either $k = 1$ or $k = 2 \lambda p^\alpha$ for some $\alpha \geq 0$ and $1 \leq \lambda \leq p - 1$. If moreover $k \leq \tfrac{n-1}{2(p-1)}$ or $\dim \bar H^k = 1$, then $\lambda$ divides $p - 1$.
\end{theorem}

Regarding the sufficient condition for proving $\lambda \mid p - 1$, notice that $\dim \bar H^k = 1$ is the situation considered in \cite{Kennard13}, and the arguments there carry over easily since this assumption implies that any element inducing periodicity not only has degree divisible by $k$ but also equals a power of a generator of $\bar H^k$.  We also note that $\lambda\mid p - 1$ holds automatically when $p = 3$, and this is useful for us later.

\begin{proof}
Let $\bar H^*$ denote the periodic subquotient. By Lemma \ref{lem:AllPeriodsIrreducible}, an element inducing periodicity does so irreducibly, so we may assume without loss of generality that $k$ is the minimum period. In particular, $k$ divides any other period as a consequence of Lemma \ref{lem:Products+Factors}.

Fix $x \in \bar H^k$ inducing periodicity on $\bar H^*$. If $k$ is odd or $x=0$, $\bar H^*=0$ is $1$-periodic. Otherwise, let $\alpha \geq 0$ denote the largest power of $p$ that divides $k$. Then we can write
	\[k = 2 \lambda p^\alpha + 2 \mu p^{\alpha + 1}\]
for some integers $\lambda$ and $\mu$ satisfying $1 \le \lambda \le p-1$ and $\mu \geq 0$. Our task is to prove $\mu = 0$ in general and that $\lambda$ divides $p - 1$ under each of the three additional conditions.

The Adem relations (e.g., see \cite[Lemma 2.2]{Kennard13}) imply a relation of the form
	\[P^{k/2} = \sum_{i \leq \alpha} P^{p^i} \circ Q_i\]
for some Steenrod algebra elements $Q_0,\ldots,Q_\alpha$. Evaluating on $x$ yields an expression of the form
	\[x^p = \sum_{i \leq \alpha} P^{p^i}(y_i)\]
for some homogeneous elements $y_0, \ldots, y_\alpha \in \bar H^*$. By Lemma \ref{lem:Products+Factors}, $x^p$ induces periodicity in $\bar H^*$. By Lemma \ref{lem:AllPeriodsIrreducible}, this periodicity is irreducible, so we have that some $P^{p^i}(y_i)$ induces periodicity in $\bar H^*$.

Using periodicity, we may write $y_i = x^s z$ for some $s \geq 0$ and $z \in \bar H^*$ with $1 \leq |z| \leq k$. Applying $P^{p^i}$ to both sides of this equation and using the Cartan formula, irreducibility, and Lemma \ref{lem:Products+Factors}, we obtain two elements of the form $P^{p^i - j}(x^s)$ and $P^j(z)$ for some $0 \leq j \leq p^i$ that induce periodicity.

{\bf Case 1}: Assume $|P^j(z)| \leq \tfrac{n-1}{p}$ or $k \leq \tfrac{n-1}{2(p-1)}$. By Lemma \ref{lem:partialP^i-module}, $z$ induces periodicity. Hence both $z$ and $P^j(z)$ have degree divisible by $k$, and this means that $k$ divides $|P^j| = 2(p-1) j$. Since $j \leq p^\alpha$, this implies that $j = p^\alpha$, $\mu = 0$, and $\lambda|p-1$, as claimed.

{\bf Case 2}: Assume $|P^j(z)| > \tfrac{n-1}{p}$ and $k > \tfrac{n-1}{2(p-1)}$. Since $|P^j(z)|$ is divisible by $k$, we have $|P^j(z)| \geq 2k$. On the other hand,
	\[|P^j(z)| = 2(p-1)j + |z| \leq 2(p-1)p^\alpha + k,\]
so $k \leq 2(p-1)p^\alpha$. This already implies $\mu = 0$. 

To finish the proof, it suffices to show $\lambda \mid p - 1$ in Case 2 when additionally $\dim \bar H^k = 1$.

First we assume that $g = \gcd(\lambda, p - 1)$ is not equal to one. Since both $P^{p^i - j}(x^s)$ and $P^j(z)$ induce periodicity, their degrees are divisible by $k$ and hence by $p^\alpha$. Hence $j = p^\alpha$. We claim moreover that $\iota(z) = p^\alpha$. Indeed, if $P^{j'}(z)$ induced periodicity for some $j' < p^\alpha$, then $k$ and hence $p^\alpha$ would divide $|P^j(z)| - |P^{j'}(z)| = 2 (p - 1) (j - j')$, a contradiction. By the additivity of $\iota$, we have $\iota(z^{\lambda/g}) = \tfrac \lambda g p^\alpha$. On the other hand,
	\[|z^{\lambda/g}| = \frac \lambda g\of{rk - 2(p-1)p^\alpha} = \of{\frac \lambda g r - \frac{p-1}{g}} k,\]
so $z^{\lambda/g}$ has degree divisible by $k$. By the assumption $\dim \bar H^k = 1$, the element $z^{\lambda/g}$ is a multiple of a power of $x$. It is a non-zero multiple since $z^{\lambda/g}$ maps to a (non-zero) element inducing periodicity under $P^{(\lambda/g)p^\alpha}$. Therefore $z$ induces periodicity and hence has degree divisible by $k$ by Lemma \ref{lem:Products+Factors}. Since $P^{p^\alpha}(z)$ also has degree divisible by $k$, the degree $|P^{p^\alpha}| = 2p^\alpha(p-1)$ is divisible by $k = 2 p^\alpha \lambda$ and hence $\lambda \mid p - 1$. 

We may now assume $\gcd(\lambda, p - 1) = 1$. We may also assume $\lambda > 1$ since otherwise $\lambda \mid p - 1$. We proceed as in the beginning of the proof but using a different Adem relation. Specifically, let $m$ be the unique integer such that $0 < m < \lambda$ and 
	\[m(p-1) = -1 + l\lambda\]
for some $0 < l < p - 1$. The Adem relations (e.g., see \cite[Lemma 2.2]{Kennard13}) imply another relation of the form
	\[P^{k/2} = P^{mp^\alpha} \circ Q_\alpha' + \sum_{i < \alpha} P^{p^i} \circ Q_i'\]
for some Steenrod algebra elements $Q_0', \ldots, Q_\alpha'$. Evaluating on $x$ as before, we obtain an element of the form $P^{mp^\alpha}(y)$ or $P^{p^i}(y)$ with $i < \alpha$ that induces periodicity. For elements of the latter type, the proof above easily implies a contradiction. For elements of the former type, we note that $y = x^s z$ for some $z \in \bar H^*$ with $|z| = 2 p^\alpha$, that $P^j(z)$ induces periodicity for some $j$, and again that $j$ is divisible by $p^\alpha$. It follows that $j = 0$, $j = p^\alpha$, or $j \geq 2p^\alpha$, which imply respectively that $P^j(z)$ equals $z$, $z^p$, or $0$. The last case cannot occur since $P^j(z)$ induces periodicity, and either of the first two cases implies that $z$ itself induces periodicity. As before, it follows that $|z| = 2 p^\alpha$ is divisible by $k = 2 \lambda p^\alpha$, so $\lambda$ equals one and hence divides $p - 1$.
\end{proof}

\section{Decomposition Lemma}\label{sec:DecompositionLemma}
In this section, we analyze the non-irreducible case of periodicity. The main result is Lemma \ref{lem:Decomposition}, which shows that any periodic cohomology ring decomposes as a sum of irreducibly periodic cohomology rings. As before, we let $H^*(M; \bZ_p)$ be $k$-periodic on degrees $1 \leq * \leq n - 1$, we fix $x \in H^k(M; \bZ_p)$ inducing this periodicity, and we define the subquotient $\bar H^*$ and the isomorphism $\xi:\bar H^i \to \bar H^{i+k}$ given by $\xi(y) = x y$ as in Equation \eqref{eqn:xi}.

Let $\End_0(\bar H^*)$ denote the ring of degree-preserving linear operators on $\bar H^*$. We define a representation
	\[\rho:\bar H^k \to \End_0(\bar H^*)\]
where the image $\rho_a$ of $a \in \bar H^k$ is defined on homogeneous elements $u \in \bar H^*$ by 
	\[\rho_a(u) = \left\{\begin{array}{rcl} 
	  		\xi^{-1}(a \cup u) & \mathrm{if} & |u| \leq n - 1 - k\\
			a \cup \xi^{-1}(u) & \mathrm{if} & |u| \geq 1 + k
			\end{array} \right.\]

\begin{lemma}
If $3k \leq n - 1$, then $\rho:\bar H^k \to \End_0(\bar H^*)$ is an injective homomorphism of rings, where the product structure on $\bar H^k$ is given by the formula $ab = \xi^{-1}(a \cup b)$.
\end{lemma}

\begin{proof}
First, we show that $\rho_a$ is well defined on the overlapping degrees $1 + k \leq |u| \leq n - 1 - k$. Indeed, the formulas agree if their images under $\xi$ agree since this map is injective in these degrees. Hence it is equivalent to prove that
	\[a \cup u = x \cup\of{a \cup \xi^{-1}(u)}.\]
By associativity and commutativity of cup products, the right-hand side simplifies to $a \cup \of{x \cup \xi^{-1}(u)} = a \cup u$, as needed. Recall for this that if $p \neq 2$ and $k$ is odd, $\bar H \equiv 0$, in which case all claims are trivial.

Second, additivity holds because cup products are bilinear.

Next, we prove multiplicativity. Fix $a, b \in \bar H^k$ and a homogeneous element $u \in \bar H^*$. We need to prove $\rho_a(\rho_b(u)) = \rho_{ab}(u)$. We do this in three cases, using the assumption that $3k \leq n - 1$.
	\begin{itemize}
	\item If $1 \leq |u| \leq k$, then $\xi^2:\bar H^{|u|} \to \bar H^{|u| + 2k}$ is an isomorphism. Therefore the claim is equivalent to the identity
		\[\xi^{2}\of{\xi^{-1}\of{a \cup \xi^{-1}(b \cup u)}} = \xi^{2}\of{\xi^{-1}\of{\xi^{-1}(a \cup b) \cup u}}.\]
	Since $\xi$ is given by multiplication by $x$, this identity is easily verified using associativity and commutativity of cup products.
	\item If $1 + k \leq |u| \leq 2k$, then we still have that $\xi:\bar H^{|u|} \to \bar H^{|u| + k}$ is an isomorphism and hence that the claim is equivalent to the identity
		\[\xi\of{a \cup \xi^{-1}\of{b \cup \xi^{-1}\of{u}}} = \xi\of{\xi^{-1}\of{a \cup b} \cup \xi^{-1}(u)},\]
	which again is easily verified.
	\item If $1 + 2k \leq |u| \leq n - 1$, then we use the following alternative formula for $\rho_a$:
		\[\rho_b(u) = b \cup \xi^{-1}(u) = b \cup \xi\of{\xi^{-2}(u)} = \xi\of{b \cup \xi^{-2}(u)}.\]
	Using formulas similar to the right-hand side of this expression for $\rho_a$ and $\rho_{ab}$, our claim is equivalent to
		\[\xi\of{a \cup \xi^{-1}\of{b \cup \xi^{-2}(u)}} = \xi\of{\xi^{-1}(a \cup b) \cup \xi^{-2}(u)},\]
	which again is easily verified.
	\end{itemize}
Finally, $\rho$ is injective since evaluation of $\rho_a$ on $x$ yields $a$ for all $a \in \bar H^k$.
\end{proof}

\begin{lemma}[Decomposition Lemma]\label{lem:Decomposition}
Let $3k\leq n - 1$, and assume $x \in \bar H^k$ induces periodicity in $\bar H^*$. There exist $x_1, \ldots, x_m \in \bar H^k$ and non-trivial subalgebras $\bar H^*_{x_1}, \ldots, \bar H^*_{x_m}$ such that all of the following hold:
	\begin{enumerate}
	\item $\bar H^* = \bar H^*_{x_1} \oplus \ldots \oplus \bar H^*_{x_m}$.
	\item Multiplication by $x_i$ induces periodicity when restricted to $\bar H^*_{x_i}$ and induces zero when restricted to $\bar H^*_{x_j}$ for all $j \neq i$.
	\item The periodicity in (2) is irreducible in the sense that, if $x_i = a + b$, then $a$ or $b$ also induces periodicity on $\bar H^*_{x_i}$.
	\end{enumerate}
\end{lemma}

Note that (1) and (2) already hold with $m = 1$ and $x_1 = x$, but (3) might fail since $x$ might decompose as a sum of elements that do not induce periodicity. An example of such a phenomenon occurring is when $\bar H^* = H^{1 \leq * \leq 2n-1}(M; \bZ_p)$ where $M^{2n}$ is a connected sum of $\CP^n$ and when $x \in \bar H^2$ corresponds to the sum of generators $x_i$ of $H^2(\CP^n;\bZ_p) \cong \bZ_p$ for each of the summands $\CP^n$. Indeed, this example models the conclusions of the lemma since the $x_i$ induce periodicity on the subspaces $\bar H^*_{x_i} = H^{1 \leq * \leq 2n-1}(\CP^n;\bZ_p)$ corresponding to the summands.

\begin{proof} 
Let $\rho:\bar H^k \to \End_0(\bar H^*)$ be the representation of rings defined above in terms of $x$. Note that $\rho_x$ is the identity map $I$. The proof shows that $\bar H^*_{x_i}$ is the image of $\rho_{x_i}$, and so $x_i$ induces periodicity on $\bar H^*_{x_i}$ if and only if $\rho_{x_i}$ is injective when restricted to $\bar H^*_{x_i}$. The lemma follows if we prove the existence of $x_1, \ldots, x_m \in \bar H^k$ such that the following hold:
	\begin{enumerate}
	\item[(1')] The subspaces $\bar H^*_{x_i} = \rho_{x_i}\of{\bar H^*}$ are non-zero and satisfy $\bar H^* = \bar H^*_{x_1} \oplus \ldots \oplus \bar H^*_{x_m}$.
	\item[(2')] The maps $\rho_{x_i}$ are the identity on $\bar H^*_{x_i}$ and zero on $\bar H^*_{x_j}$ for all $j \neq i$.
	\item[(3')] Whenever $x_i = a + b$, the restriction of $\rho_a$ or $\rho_b$ to $\bar H^*_{x_i}$ is injective.
	\end{enumerate}

Fix elements $x_1, \ldots, x_m \in \bar H^k$ satisfying (1') and (2') such that $m$ is maximal. We claim that this maximality forces (3'). Note here that we may assume $m\ge 1$, since otherwise $\bar H = \bar H_x$ must fail (1'), i.e. be zero, and the conclusion of the lemma holds with $m=0$ trivially.

If (3') fails on some $\bar H^*_{x_i}$, then there is a decomposition $x_i = a + b$ such that both $\rho_a$ and $\rho_b$ fail to be injective when restricted to $\bar H^*_{x_i}$. Our task is to modify the decomposition to obtain another one $x_i = \tilde a + \tilde b$ such that the elements $x_1, \ldots, x_{i-1}, \tilde a, \tilde b, x_{i+1}, \ldots, x_m$ satisfy (1') and (2'), a contradiction to maximality.

The main step is constructing $\tilde a$ and $\tilde b$. Consider the maps $\rho_a$ and $\rho_b$. We first claim there exists an integer $l \geq 1$ such that $\rho_{a_1}$ and $\rho_{b_1}$ are semi-simple, where $a_1 = a^{p^l}$ and $b_1 = b^{p^l}$. Indeed, as endomorphisms of a vector space over $\bZ_p$, there exists a Jordan-Chevalley decomposition
	\[\rho_a = s_a + n_a\]
into semi-simple and nilpotent parts that commute and have the property that each is equal to a polynomial in $\rho_a$. Working over $\bZ_p$, we have that
	\[\rho_a^{p^l} = s_a^{p^l} + n_a^{p^l}.\]
Since $n_a$ is nilpotent, $n_a^{p^l} = 0$ for all sufficiently large $l$. By Lemma \ref{lem:ss}, $s_a^{p^l}$ is semi-simple for all $\ell$, so the claim holds for $\rho_a$. The same argument works for $\rho_b$, possibly after increasing $l$. Defining $a_1$ and $b_1$ as above and recalling that $\rho$ is a homomorphism of rings, the claim follows.

Second, we use the failure of (3') to decompose $\bar H^*_{x_i}$ as a direct sum. Recall that $\rho_a + \rho_b = I$ on this subspace. Raising to the power $p^l$, we have
	\[\rho_{a_1} + \rho_{b_1} = I\]
on $\bar H^*_{x_i}$. 
Let $K_a = \ker(\rho_{a_1}) \cap \bar H^*_{x_i}$ and $K_b = \ker(\rho_{b_1}) \cap \bar H^*_{x_i}$. 
Since these operators are semi-simple, 
their restrictions to $\bar H^*_{x_i}$ are semi-simple by Lemma \ref{lem:ss}, Part \eqref{lem:ssPart1}. 
Therefore we may choose a $\rho_{a_1}$-invariant complement $K_a^c$ to $K_a$ in $\bar H^*_{x_i}$ and a $\rho_{b_1}$-invariant complement $K_b^c$ to $K_b$ in $\bar H^*_{x_i}$. Notice that
	\[K_a \subseteq K_b^c\]
since, for $u \in K_a$, if $u = u_{K_b} + u_{K_b^c}$ is a decomposition into components $u_{K_b} \in K_b$ and $u_{K_b^c} \in K_b^c$, then
	\[u 	= \of{\rho_{a_1} + \rho_{b_1}}(u) 
		= \rho_{b_1}(u)
		= \rho_{b_1}(u_{K_b^c}) \in K_b^c,\]
where in the last equality we used the fact that $K_b^c$ is $\rho_{b_1}$-invariant. The condition $K_a \subseteq K_b^c$ implies we get a decomposition
	\[\bar H^*_{x_i} = K_a \oplus \of{K_a^c \cap K_b^c} \oplus K_b\]
by Lemma \ref{lem:DirectSum}, Part \eqref{lem:DirectSumPart2}. Note for this that $K_a \cap K_b = 0$ since $I = \rho_{a_1} + \rho_{b_1}$ on $\bar H^*_{x_i}$, so we are applying the lemma with $U = K_a \oplus K_b$.

Third, we look at the action of $\rho_{a_1}$ and $\rho_{b_1}$ on the above decomposition of $\bar H^*_{x_i}$ and further modify the maps. Since $I = \rho_{a_1} + \rho_{b_1}$ on $\bar H^*_{x_i}$, notice that all three direct summands are both $\rho_{a_1}$- and $\rho_{b_1}$-invariant. Together these facts imply the following block diagonal forms:
	\[
	\rho_{a_1} = \of{\begin{array}{ccc} 0 & 0 & 0 \\ 0 & G & 0 \\ 0 & 0 & I\end{array}}
		\hspace{.2in} \mathrm{and} \hspace{.2in}
	\rho_{b_1} = \of{\begin{array}{ccc} I & 0 & 0 \\ 0 & I-G & 0 \\ 0 & 0 & 0\end{array}},
	\]
for some endomorphism $G$ of the (possibly trivial) subspace $K_a^c \cap K_b^c$. By definition of $K_a$ and $K_b$, both $G$ and $I - G$ are invertible. Since we are working over $\bZ_p$, $G$ has finite order and so there exist $r, s \geq 1$ such that $G^r = I$ and $(I - G)^s = I$. Setting
	\[a_2 = a_1^{r} \hspace{.2in}\mathrm{and}\hspace{.2in} b_2 = b_1^{s} - a_1^{r} b_1^{s},\]
the actions of $\rho_{a_2}$ and $\rho_{b_2}$ on $\bar H^*_{x_i}$ are given by
	\[
	\rho_{a_2} = \of{\begin{array}{ccc} 0 & 0 & 0 \\ 0 & I & 0 \\ 0 & 0 & I\end{array}}
		\hspace{.2in} \mathrm{and} \hspace{.2in}
	\rho_{b_2} = \of{\begin{array}{ccc} I & 0 & 0 \\ 0 & 0 & 0 \\ 0 & 0 & 0\end{array}}.
	\]

Finally, we modify the decomposition once more to control the action on $\bar H^*_{x_j}$ with $j \neq i$. To do this, we define
	\[\tilde a = x_i a_2 x_i \hspace{.2in} \mathrm{and} \hspace{.2in} \tilde b = x_i b_2 x_i\]
and note that the subspaces
	\[\bar H^*_{\tilde a} = (K_a^c \cap K_b^c) \oplus K_b \hspace{.2in} \mathrm{and} \hspace{.2in}
	  \bar H^*_{\tilde b} = K_a\]
together with the $\bar H^*_{x_j}$ with $j \neq i$ satisfy (1') and that the maps $\rho_{\tilde a}$ and $\rho_{\tilde b}$ satisfy (2'), so the proof is complete.
\end{proof}


\section{Partial Four Periodicity theorem}\label{sec:Partial4PT}

In this section, we complete the proof of the analogues of the $\Z_p$- and Four Periodicity Theorems. 
\begin{theorem}[Partial $\Z_p$-periodicity theorem]
Assume $H^*(M; \bZ_2)$ is $k$-periodic on degrees $1 \leq * \leq n - 1$ for some prime $p$. 
	\begin{enumerate}
	\item If $p = 2$, then the minimum period is of the form $2^a$.
	\item If $p = 3$, then the minimum period is either $1$ or of the form $2\lambda \cdot 3^\alpha$ for some $\lambda \in \{1,2\}$.
	\item If $p \geq 5$ and $2(p - 1)k \leq n - 1$, then the minimum period is either $1$ or of the form $2 \lambda p^\alpha$ for some divisor $\lambda$ of $p - 1$.
	\end{enumerate}
\end{theorem}

Recall that $1$-periodicity for $p \geq 3$ is equivalent to the groups $H^i(M; \bZ_p)$ vanishing for $1 < i < n - 1$. This does not occur in the generic situation where the element inducing periodicity is both non-zero and of degree at least two.

\begin{proof}
Suppose first that $p = 2$, and define the subquotient $\bar H^*$ as in Section \ref{sec:Subquotient}. By definition of periodicity, we may fix and element $x \in \bar H^k$ that induces periodicity and has $3k \leq n - 1$.

The Decomposition Lemma implies the existence of $x_1, \ldots, x_m \in \bar H^k$ and non-trivial subalgebras $\bar H^*_{x_1}, \ldots, \bar H^*_{x_m}$ such that all of the following hold:
	\begin{enumerate}
	\item $\bar H^* = \bar H^*_{x_1} \oplus \ldots \oplus \bar H^*_{x_m}$.
	\item Multiplication by $x_i$ induces periodicity when restricted to $\bar H^*_{x_i}$ 
		and induces zero when restricted to $\bar H^*_{x_j}$ for all $j \neq i$.
	\item The periodicity in (2) is irreducible in the sense that, if $x_i = a + b$, 
		then $a$ or $b$ also induces periodicity on $\bar H^*_{x_i}$.
	\end{enumerate}
If $m=0$, then $\bar H^*$ vanishes and is trivially $1$-periodic. Otherwise, we apply the $\Z_p$-Periodicity Theorem to each $\bar H^*_{x_i}$, which is irreducibly periodic. Thus there exist $y_i \in \bar H^*_{x_i}$ with degree of the form $2^{\alpha_i}$ such that $y_i$ induces periodicity in $\bar H^*_{x_i}$. Fix $\alpha = \max\{\alpha_i\}$, and define $z_i$ to be the power of $y_i$ that has degree $2^\alpha$. Finally, set $z = \xi^{-1}(x_1 z_1) + \ldots + \xi^{-1}(x_m z_m)$. By (1) and (2), the element $z$ induces periodicity in $\bar H^*$. Since the minimum period divides any period, this shows that the minimum period is a power of two.

The proof for $p \geq 3$ is analogous. We may assume $k \geq 2$ and is minimal without loss of generality, so we have that $k$ is even and that $3k \leq n - 1$ by the definition of periodicity. The Decomposition Lemma therefore applies, and the minimum periods of the irreducible summands $\bar H^*_{x_i}$ will be of the form $2 \lambda_i p^{\alpha_i}$ where $\lambda_i$ divides $p - 1$. This time, we take powers to obtain $y_i \in \bar H^*_{x_i}$ with degree $2 (p - 1) p^\alpha$ where $\alpha = \max\{\alpha_i\}$, noting that $\lambda_i$ divides $p - 1$ for all $i$. Finally the minimum period divides $2 (p - 1) p^\alpha$ and hence is of the desired form.
\end{proof}

Following the strategy of the proof of \cite[Theorem C]{Kennard13}, we combine integral periodicity with finer $\Z_p$ periodicity to obtain rational Four Periodicity.

\begin{theorem}[Partial Four Periodicity Theorem]\label{thm:P4PT}
Assume $M^n$ is a closed, orientable $n$-manifold such that $H^*(M^n; \bZ)$ is $k$-periodic on degrees $1 \leq * \leq n -1$ and satisfies $3k \leq n - 1$. If moreover $H^1(M; \Z_p) = 0$ for $p = 2$ and $p = 3$, 
then $H^*(M; \bQ)$ is $\gcd(4,k)$-periodic on degrees $1 \leq * \leq n - 1$.
\end{theorem}

\begin{proof}   
Let $x \in H^k(M;\Z)$ be an element inducing periodicity. If $x$ is torsion, then periodicity implies that all cohomology groups are torsion in degrees from $2$ to $n-2$. In degrees $1$ and $n - 1$, the same holds by the assumption $H^1(M; \Z_2) = 0$ together with Poincar\'e duality. In particular, $M$ is a rational homology sphere. Similarly if the image $x_p \in H^k(M; \Z_p)$ under the homomorphism induced by the reduction modulo $p$ on the level of coefficients is zero for some $p \in \{2, 3\}$, then $M$ is a $\Z_p$-homology sphere and hence a rational homology sphere by the Universal Coefficients Theorem. We may assume that $x_p \neq 0$ for $p \in \{2,3\}$. 

By the relationship between cup products and cap products, we see that capping with $x$ also induces isomorphisms on homology between degrees $1$ and $n - 1$. By the Universal Coefficients Theorem, it follows that cupping with $x_p$ induces periodicity on $\Z_p$-cohomology in degrees from $1$ to $n-1$. Since $3k \leq n - 2$, the $\Z_p$ Periodicity Theorem implies that some $y_p$ induces $k_p$-periodicity in $H^*(M; \Z_p)$ for $p \in \{2,3\}$, where $k_2$ is of the form $\gcd(2^a,k)$ and $k_3$ is of the form $\gcd(4,k)\cdot 3^b$.

By periodicity and the assumption $H^1(M; \Z_p) = 0$, we have $H^{1+k_p}(M; \Z_p) = 0$. By the Bockstein exact sequence associated to the short exact sequence $0 \to \Z \to \Z \to \Z_p \to 0$, there exists a lift $\tilde y_p$ of $y_p$ under the reduction homomorphism. Moreover, multiplication by $\tilde y_p$ induces maps on $H^*(M; \Z_p)$ in degrees from $1$ to $n-1$ that have finite kernel and cokernel. Passing to rational coefficients, we obtain $k_2$- and $k_3$-periodicity in the rational cohomology of $M$ in degrees from $1$ to $n-1$. Since the minimum period divides any period, we find that the rational cohomology of $M^n$ is $\gcd(4, k)$-periodic on degrees $1$ to $n-1$, as claimed.
\end{proof}

Before we prove the theorem, we note an important consequence for our main results.

\begin{corollary}\label{cor:bodd}
If $N^{n-k} \to M^n$ is a $(n - k - 1)$-connected inclusion of closed manifolds with finite fundamental group and $3k\leq n - 1$, then the universal cover of $M$ has four-periodic rational cohomology on degrees $1$ to $n-1$. Moreover, if $n \equiv 0 \bmod 4$ or if $n$ is even and $k \equiv 2 \bmod 4$, then the odd Betti numbers of $M$ vanish. Similar statements hold for $N$ and its universal cover.
\end{corollary}

The proof of the corollary follows from Theorem \ref{thm:P4PT} and the Periodicity Lemma from \cite{Wilking03}, which we state here:

\begin{theorem}[Periodicity Lemma]\label{thm:PT}
If $N^{n-k} \to M^n$ is an $(n - k - l)$-connected inclusion of closed, orientable manifolds with $n - k - 2l > 0$, then there exists $e \in H^k(M;\Z)$ such that the map $H^i(M;\Z) \to H^{i+k}(M;\Z)$ induced by multiplication by $e$ is surjective for $l \leq i < n - k - l$ and injective for $l < i \leq n - k - l$.
\end{theorem}

\begin{proof}[Proof of Corollary]
The map $N^{n-k} \to M^n$ lifts to a map $\tilde N^{n-k} \to \tilde M^n$ of universal covers. Covering projections induce isomorphisms on higher homotopy groups, so the latter map is also $(n - k - 1)$-connected. Additionally, $\tilde N$ and $\tilde M$ are simply connected manifolds that are also compact because the (isomorphic) fundamental groups are finite. Therefore $H^*(\tilde M; \Z)$ is $k$-periodic from degree $1$ to $n-1$ by the Periodicity Lemma, and Theorem \ref{thm:P4PT} implies that $\tilde M$ has $\gcd(4,k)$-periodic rational cohomology. The conclusion for $\tilde N$ follows because of the connectedness of the map $\tilde N \to \tilde M$. This completes the proof of the first claim.

Now suppose $k \equiv 2 \bmod 4$ or $n \equiv 0 \bmod 4$. In the former case, all Betti numbers are bounded above by $\dim H^1(\tilde M; \bQ) = 0$, so all of them vanish. In the latter case, we actually recover the same fact since $\dim H^3(M;\bQ) = H^{n-3}(M;\bQ)$, which is again bounded above by $\dim H^1(M; \bQ)$ by Four Periodicity. Therefore $\tilde M$ and $\tilde N$ have vanishing odd Betti numbers. This fact descends to $M$ and $N$ since in general a covering projection $p:\tilde X \to X$ realized by a free action by a finite group $\pi$ has the property that $p^*$ induces an injection $H^*(X; \bQ) \to H^*(\tilde X; \bQ)$ whose image is the subset of elements fixed by the induced action by $\pi$.
\end{proof}

\section{Proof of Theorem \ref{thm:t8}}\label{sec:t8}

In this section, we combine the Partial Four Periodicity Theorem with Wilking's Connectedness Lemma and the $\gS^1$ Splitting Theorem in \cite{KennardWiemelerWilking21preprint} to prove Theorem \ref{thm:t8}. Since the first of these results is used later as well, we state it here.

\begin{theorem}[Connectedness Lemma for $\Ric_2$]\label{thm:CL} 
Let $M^n$ be a closed Riemannian manifold with $\Ric_2 > 0$.
    \begin{enumerate}
        \item If $N^{n-k} \subseteq M^n$ is a fixed-point component of an isometric circle action on $M^n$, then the inclusion is $(n-2k+1)$-connected.
        \item If $N_1^{n-k_1}$ and $N_2^{n-k_2}$ are two such submanifolds with $k_1 \leq k_2$, then the inclusion $N_1 \cap N_2 \subseteq N_2$ is $(n - k_1 - k_2 - 1)$-connected.
    \end{enumerate}
\end{theorem}

Theorem \ref{thm:CL} is essentially due to Wilking \cite{Wilking03}, and the precise statement of (a) in the presence of symmetry is proved in \cite[Theorem 1.7]{Mouille22b}, building on work of Guijarro and Wilhelm \cite{GuijarroWilhelm22}. 

Notice that $\dim(N_1 \cap N_2) = n - k_1 - k_2$ in the second part of the Connectedness Lemma when the submanifolds intersect transversely. In this case, the inclusion is therefore $(\dim(N_1 \cap N_2) - 1)$-connected, and the Periodicity Lemma implies that $H^*(N_2;\Z)$ is $k_1$-periodic from degree $1$ to $\dim(N_2) - 1$ as long as $3k_1 \leq \dim(N_2) - 1$.

The proof of Theorem \ref{thm:t8} also requires the following result: 

\begin{lemma}\label{lem:t4}
Let $M^n$ be an even-dimensional, closed, orientable Riemannian manifold with $\Ric_2 > 0$ and $\gT^4$ symmetry. Let $F^f \subseteq M^{\gT^4}$ be a fixed-point component. If $f \equiv 0 \bmod 4$ and the isotropy representation of $\gT^4$ on the normal space to $F^f$ has exactly four weights, then $H^{odd}(F;\bQ) = 0$.
\end{lemma}

This lemma is analogous to a result that is stated in the introduction of Section 3 of \cite{KennardWiemelerWilking21preprint} for manifolds with positive sectional curvature and was improved in \cite[Lemma 3.1]{Nienhaus-pre} for fixed-point components of $\gT^3$-actions with only three weights in the isotropy representation. 

\begin{proof}
The result is trivial if $f = 0$, so we may assume $f \geq 4$. 

Let $k_1 \leq k_2 \leq k_3 \leq k_4$ denote twice the multiplicities of the weights of the isotropy representation, and let $N_i^{n - k_i}$ denote the fixed-point components containing $F$ of the four corresponding circle factors of $\gT^4$.

Since these codimensions are even, we can choose distinct $i, j \in \{2, 3, 4\}$ such that $k_i + k_j \equiv 0 \bmod 4$. Fix $h \in \{2, 3, 4\} \setminus \{i, j\}$. 

We look at the transverse intersection of $N_1$ and $N_h$ in $M$. It has dimension $f + k_i + k_j$, which we note is divisible by four and is at least $2 k_1 + 4$. The Connectedness and Periodicity Lemmas imply that $N_h$ is $k_1$-periodic on degrees $1 \leq * \leq \dim(N_h) - 1$. Since $\dim N_h \geq 3k_1 + 4$, the Partial Four Periodicity Theorem applies and we see that $N_h$ has four-periodic rational cohomology from degree $1$ to $\dim(N_h) - 1$. By the Connectedness Lemma, $N_1 \cap N_h$ is four-periodic from degrees $1$ to $\dim(N_1 \cap N_h) - 1$. Since the latter manifold has dimension divisible by four, Poincar\'e duality implies that its third Betti number is at most its first and thus vanishes since manifolds of dimension larger than $2$ with $\Ric_2 > 0$ have positive Ricci curvature and hence finite fundamental group. Therefore $N_1 \cap N_h$ has vanishing odd Betti numbers. Since $F$ is a fixed-point component of a torus action on $N_1 \cap N_h$, the same conclusion holds for $F$ (see for example \cite[Theorem 1.12]{KennardWiemelerWilking21preprint}).
\end{proof}

\begin{proof}[Proof of Theorem \ref{thm:t8}]
We are given a $\gT^{8}$-action on a closed Riemannian manifold $M$ with $\Ric_2 > 0$ and dimension divisible by four, and we are given a fixed point component $F^f$ of the $\gT^8$-action.

We claim we can pass to a subtorus $\gT^7$ whose fixed-point component $G$ containing $F$ has $\dim G \equiv 0 \bmod 4$. If $F$ already has this divisibility property, then this is achieved by choosing a $\gT^7$ whose fixed-point set is the same as $\gT^8$. Otherwise $F$ has codimension congruent to two modulo four, and we argue as follows. The Borel formula states that the sum of the codimensions of $F \subseteq G$ over all fixed-point components of $\gT^7$ subgroups of $\gT^8$ equals the codimension of $F \subseteq M$ (see \cite[Theorem 5.3.11]{AlldayPuppe93}). Since the latter is congruent to two modulo four, some term in the sum must also be congruent to two modulo four. This proves the claim.

Next, we apply the $\gS^1$-Splitting Theorem three times to obtain a three-dimensional subgroup $\gH \subseteq \gT^7$ whose fixed-point set $N$ containing $G$ has the property that the induced action by $\gT^7/\gH \cong \gT^4$ splits as a product of four circle representations. (This is the corollary to Theorem F in \cite{KennardWiemelerWilking21preprint}.) This is exactly the setting of Lemma \ref{lem:t4} (with $F$ replaced by $G$), and so we conclude that $H^{odd}(G; \bQ) = 0$. Since $F$ is a fixed-point component of a torus action on $G$, we also have $H^{odd}(F; \bQ) = 0$, as required.
\end{proof}

\section{Proof of Theorem \ref{thm:tcig}}\label{sec:t9}

In this section, we prove Theorem \ref{thm:tcig}. As in the previous section, we first need a lemma that extends a result for positive sectional curvature to the case of $\Ric_2 > 0$. 

\begin{lemma}\label{lem:t3}
Let $M^n$ be a closed, orientable Riemannian manifold with $\Ric_2 > 0$ and $\gT^3$ symmetry. Let $F^f \subseteq M^{\gT^3}$ be a fixed-point component. If $f \geq \tfrac{n+1}{3}$, then $H^*(F;\bQ)$ is four-periodic from degree $1$ to $f-1$.
\end{lemma}


\begin{proof}
Since the conclusion is about $F$, we may replace $M$ by the fixed-point component containing $F$ of a maximal finite isotropy group of the isotropy representation. In this way, we preserve the lower bound on $f$ while gaining the property that $\gT^3$ acts with connected isotropy groups near $F$. 

We apply the refinement of the $\gS^1$-Splitting Theorem stated in \cite[Lemma 2.2]{KennardWiemelerWilking21preprint} to the isotropy representation of $\gT^3$ on the normal space to $F$. It states that one of the following occurs:
\begin{enumerate}
    \item[(a)] There is a one-dimensional subgroup $\gH \subseteq \gT^3$ whose fixed-point component $N^m$ containing $F$ has the property that the induced action by $\gT^3/\gH \cong \gT^2$ splits as a product of circle actions and satisfies $\cod N \geq \tfrac 1 2 \cod (F\subseteq N)$. 
    \item[(b)] The representation already splits as a product representation of $\gS^1 \times \gT^2$.
\end{enumerate}

If Case (a) occurs, let $L_1^{m-l_1}$ and $L_2^{m - l_2}$ denote the fixed-point components containing $F$ of these two circle actions on $N$, and assume $l_1 \leq l_2$. By the assumption on $f$ and the choice of $N$, we have
	\[3f - 1 \geq n = f + l_1 + l_2 + \cod N \geq f + 2(l_1 + l_2) \geq f + 4 l_1.\]
Since $2f$ and $l_1$ are even, this implies $2f - 2 \geq 4l_1$ and hence 
    \[\dim L_2 - 1
    = f - 1 + l_1
    \geq 3l_1.\]
The Connectedness and Periodicity Lemmas and the Partial Four Periodicity Theorem now imply that $L_2$ has four-periodic rational cohomology from degrees $1$ to $\dim L_2 - 1$. Since the inclusion $F \to L_2$ is $(\dim L_2 - l_1 - 1)$-connected, the conclusion of the lemma holds.

Now suppose Case (b) occurs, and assume furthermore that the $\gT^3$ splits entirely as a product of three circle representations. Let $N_1^{n-k_1}$, $N_2^{n-k_2}$, and $N_3^{n-k_3}$ denote the fixed-point components of these circle factors, and assume $k_1 \leq k_2 \leq k_3$ without loss of generality. Notice that $3k_1 \leq \dim N_3 - 1$, since otherwise we have
	\[n \leq 3k_1 + k_3 \leq \tfrac 4 3(k_1 + k_2 + k_3) = \tfrac 4 3 (n - f),\]
which contradicts our assumption on $f$. Hence the Connectedness and Periodicity Lemmas and the Partial Four Periodicity theorem apply to the transverse intersection of $N_1$ and $N_3$ and shows that $N_3$ has four-periodic rational cohomology on degrees from $1$ to $\dim N_3 - 1$. The Connectedness Lemma now shows that the inclusions $F \to N_2 \cap N_3 \to N_3$ are $f$-connected, so it follows that $F$ is four-periodic from degree $1$ to $f - 1$, as claimed. 

Finally, we finish the proof in Case (b) when the representation splits as $\gS^1 \times \gT^2$ and moreover the $\gT^2$ representation has three weights. Let $m_0$ and $m_1 \leq m_2 \leq m_3$ denote twice the multiplicities of the $\gS^1$ and the $\gT^2$ representations, respectively. For each $i \in \{1, 2, 3\}$, there is a circle whose fixed-point component contains transversely intersecting submanifolds of codimension $m_0$ and $m_i$ whose intersection is $F$. In particular, if $2 m_j \leq f - 1$ for any $j \in \{0, 1, 2, 3\}$, then arguments as above imply that $F$ has the desired periodicicty in its cohomology. If none of these hold, then we have
	\[n = f + m_0 + \ldots + m_4 \geq 3f,\]
in contradiction to the assumption of the lemma.
\end{proof}

\begin{proof}[Proof of Theorem \ref{thm:tcig}]
We are given a closed Riemannian manifold $M^n$ with $\Ric_2 > 0$, and we assume that $\gT^9$ acts isometrically on $M^n$ in such a way that there is a fixed point $x$ and neighborhood of $x$ in which all isotropy groups are connected.

By the main estimates proved in \cite[Theorem C]{KennardWiemelerWilking25}, there exist subgroups
	\[\gT^3 \supseteq \gT^2 \supseteq \gT^1\]
of $\gT^9$ with corresponding fixed-point components
	\[F_3 \subseteq F_2 \subseteq F_1\]
satisfying the codimension bounds 
	\[\cod\of{F_3 \subseteq F_2} \leq \tfrac{3}{10} \dim F_2,\]
	\[\cod\of{F_2 \subseteq F_1} \leq \tfrac 2 7 \dim F_1,\]
and
	\[\cod\of{F_1 \subseteq M} \leq \tfrac 1 4 n.\]
In particular, $F_3$ satisfies 
	\[\dim(F_3) \geq \ceil{\frac{3}{8}n} \geq \frac{n+1}{3}.\]

Lemma \ref{lem:t3} now implies that $F_3$ has four-periodic rational cohomology from degree $1$ to $\dim F_3 - 1$. The strategy,  as in \cite{KennardWiemelerWilking25}, is to apply the Connectedness Lemma to pull this property up to $M$. We also make repeated use of Lemmas \ref{lem:uptoc} and \ref{lem:k1plus2k2}, which are slight modifications of lemmas in \cite{KennardWiemelerWilking25} and which we state below the proof. The proof works smoothly under the additional assumptions $k_2 \geq 6$ and $k_1 \geq 6$, so we do this case first.

First, we claim that $k_2+3 \leq \dim F_3 - 1$. Indeed, this fails only if $k_2 \geq \dim F_3 - 2$. Replacing $\dim F_3$ by $\dim F_2 - k_3$, using the upper bound on $k_3$, replacing $\dim F_2$ by $\dim F_1 - k_2$, and then using the upper bound on $k_1$, this implies that $\dim F_1 \leq 8$, which contradicts the fact that $F_1$ admits an effective $\gT^8$ action with a fixed point. Therefore the claim holds, and $F_3$ is four-periodic on $1 \leq * \leq k_2 + 3$.

Second, the inclusion $F_3 \to F_2$ is $(\dim F_2 - 2k_3 + 1)$-connected by the Connectedness Lemma. By the upper bound $k_3 \leq \tfrac{3}{10} \dim F_2$, this inclusion is $(k_3+2)$-connected. Hence Lemma \ref{lem:uptoc} implies that $F_2$ is $4$-periodic on $1 \leq * \leq k_2 + 3$.

Third, the inclusion $F_2 \to F_1$ is $(\dim F_1 - 2k_2 + 1)$-connected. Estimating as before, we have $\dim F_1 - 2k_2 + 1 \geq k_2 + 2$, so this inclusion is $(k_2+2)$-connected. Lemma \ref{lem:uptoc} implies that $F_1$ is $4$-periodic on $1 \leq * \leq k_2 + 3$, and then Lemma \ref{lem:k1plus2k2} implies that $F_1$ is $4$-periodic on $1 \leq * \leq \dim F_1 - 2k_2 + 2$. Estimating as above, we find that $\dim F_1 - 2k_2 + 2 \geq k_1 + 3$, so $F_1$ is $4$-periodic on $1 \leq * \leq k_1 + 3$.

Finally, the Connectedness Lemma implies that the inclusion $F_1 \to M$ is $(n-2k_1+1)$- and hence $(k_1+2)$-connected, so $M$ is $4$-periodic on $1 \leq * \leq k_1 + 3$. By Lemma \ref{lem:k1plus2k2} and the estimate $k_1 \leq \tfrac n 4$, we see that $M$ is $4$-periodic on $1 \leq * \leq n - 1$.

This completes the proof if $k_1 \geq 6$ and $k_2 \geq 6$. To finish the proof, first suppose $k_1 \geq 6$ and $k_2 \leq 4$. We can choose another fixed-point component $F_2' \subseteq F_1$ that is not contained in $F_2$ and that has the property $\cod(F_2' \subseteq F_1) \leq \tfrac{3}{10} \dim F_1$. Since $F_2$ has codimension four in $F_1$, the intersection $F_2 \cap F_2'$ is either transverse or a codimension two fixed-point component of a circle action on $F_2'$. In either case, the Connectedness and Periodicity Lemmas imply that $F_2'$ has $4$-periodic cohomology on degrees $1 \leq * \leq \dim F_2' - 1$. By the Connectedness Lemma and Lemma \ref{lem:uptoc}, we can argue as above to conclude that $F_1$ is $4$-periodic on $1 \leq * \leq k_1 + 3$ and then continue as before.

Next suppose $k_1 \leq 4$. The case of $k_1 = 2$ follows immediately from the Connectedness and Periodicity Lemmas, so we may assume $k_1 = 4$. As in the previous paragraph, we obtain a fixed-point component $F_1' \subseteq M$ with $\cod(F_1' \subseteq M) \leq \tfrac{2}{7} \dim M$ and with $4$-periodic cohomology on $1 \leq * \leq \dim F_1' - 1$. Since the inclusion $F_1' \to M$ is $(n - 2k_1 + 1)$-connected, we have that $M$ has $4$-periodic cohomology on degrees $1 \leq * \leq n - 2k_1' + 2$. Since $k_1' \leq \tfrac n 4$, the periodicity goes up to $\lceil \tfrac n 2 \rceil + 2$. This is sufficient with Poincar\'e duality to conclude that $M$ is $4$-periodic on $1 \leq * \leq n - 1$.
\end{proof}

\begin{lemma}\label{lem:uptoc}
    Let $N \subseteq M$ be a $c$-connected inclusion of closed, orientable manifolds for some integer $c$ satisfying $0 < k < c - 1$. 
        \begin{enumerate}
            \item $H^*(N)$ is $k$-periodic on $1 \leq * \leq c$ if and only if $H^*(M)$ is as well.
            \item If $H^*(N)$ is $k$-periodic on $1 \leq * \leq c + 1$, then $H^*(M)$ is as well.
        \end{enumerate}
\end{lemma}

Part (1) is a slight modification to \cite[Lemma 4.8]{KennardWiemelerWilking25}, which considered periodicity on $0 \leq * \leq c$ instead of $1 \leq * \leq c$, and the proof is exactly the same. Part (2) is a slight but necessary improvement we need for later, but again the proof works in the same way. Specifically, we apply (1) and then consider the two additional diagrams
   \[
\xymatrix{
H^{c-k}(M) \ar[r]^{\tilde x} \ar[d] & H^{c}(M) \ar[d] \\
H^{c-k}(N) \ar[r]^{x} & H^{c}(N)
}
\quad\quad\quad\quad
\xymatrix{
H^{c+1-k}(M) \ar[r]^{\tilde x} \ar[d] & H^{c+1}(M) \ar[d] \\
H^{c+1-k}(N) \ar[r]^{x} & H^{c+1}(N)
}
\]
in the notation of \cite[Lemma 4.8]{KennardWiemelerWilking25}, and we use the surjectivity and injectivity, respectively, on $N$ to conclude the corresponding property on $M$.

\begin{lemma}\label{lem:k1plus2k2}
    Let $M^n$ be a closed, orientable, Riemannian manifold with $\Ric_2>0$. 
    Suppose $N^{n-k}$ is a fixed-point component of an isometric circle action on $M$ and satisfies $k \geq 6$.
    \begin{enumerate}
        \item If $H^*(M;\Q)$ is $4$-periodic from degree $1$ to 
        $k+3$, then $H^*(M;\Q)$ is $4$-periodic from degree $1$ to $n - 2k + 2$.       
        \item If additionally $k \leq \tfrac{n+3}{4}$, then $H^*(M;\Q)$ is $4$-periodic from degree $1$ to $n - 1$.
    \end{enumerate}
\end{lemma}

\begin{proof}
    The proof of (1) and (2) follow exactly that of \cite[Lemma 4.10]{KennardWiemelerWilking25} in the case where $g = 4$. In particular, the map $H^{k-4}(M) \to H^k(M)$ given by Four Periodicity up to $k+3$ is surjective since $k \geq 6$, so we get a factorization of the form $e = xy$ where $x$ has degree four as in the proof of \cite[Lemma 4.10]{KennardWiemelerWilking25}. 
    
    For the first claim, the only modifications to the proof required are replacing $k - 2$ by $k - 1$ in the bounds coming from the Periodicity Lemma. This explains why we require Four Periodicity up to degree $k+3$ instead of just $k+2$ and why our conclusion only goes up to $n - 2k + 2$ instead of $n - 2k+3$. This is the price for passing from positive sectional curvature to positive $\Ric_2$.

    For the second claim, the proof from \cite{KennardWiemelerWilking25} follows, with the following modifications. We have injectivity of the map $H^i(M) \to H^{i+4}(M)$ given by multiplication by $x$ for all $1 \leq i \leq n - 2k + 1$. We claim that this map is also injective for $n - 2k + 2 \leq i \leq n - 5$. In this range, we have $5 \leq n - i < n - 2k + 2$ since $k \leq \tfrac{n+3}{4}$, therefore the element $z \in H^{n-i}(M)$ in the proof factors, and the rest of the proof is the same.
%
    %
%
\end{proof}

\appendix

\section{Semi-simple operators}

In this appendix, we record a few elementary statements on the existence of direct sum decompositions and semi-simple operators for later use. In particular, Part \eqref{lem:ssPart4} of Lemma \ref{lem:ss} is used in the proof of the Decomposition Lemma (Lemma \ref{lem:Decomposition}).

Recall that, for a vector space $V$, a {\it complement} of a subspace $U$ is another subspace $U^c$ such that $U \cap U^c = 0$, $U + U^c = V$, and $\dim V = \dim U + \dim U^c$. Any two of these conditions implies the third by the equation
	\[\dim(A + B) = \dim A + \dim B - \dim(A \cap B)\]
Given subspaces $V_1$ and $V_2$ whose intersection is trivial, we write the span $V_1 + V_2$ as the {\it direct sum} $V_1 \oplus V_2$ since it is isomorphic to the abstract direct sum of the two vector spaces.

\begin{lemma}\label{lem:DirectSum}
Let $V$ be a vector space and $U$ be a subspace.
		\begin{enumerate}
	\item\label{lem:DirectSumPart1} 
		If $V = V_1 \oplus V_2$ and $V_1 \subseteq U$, then $U = (U \cap V_1) \oplus (U \cap V_2)$. 
		
		\item\label{lem:DirectSumPart2} 
		If $U = U_1 \oplus U_2$ and $U_1 \subseteq U_2^c$, 
		then $V = U \oplus (U_1^c \cap U_2^c)$ where $U_i^c$ is a complement of $U_i$ in $V$ for $i \in \{1,2\}$. 
		
		\item\label{lem:DirectSumPart3} 
		If $U \subseteq W \subseteq V$ and $U^c$ and $W^c$ are complements of $U$ in $W$ and of $W$ in $V$, 
		respectively, then $U^c \oplus W^c$ is a complement of $U$ in $V$.

	\end{enumerate}
\end{lemma}

In all parts but \eqref{lem:DirectSumPart3}, there is an additional condition present that cannot be dropped. In all parts, the proof is straightforward and therefore omitted. 

An endomorphism $f$ of a vector space $V$ is {\it semi-simple} if every invariant subspace has an invariant complement. This means that if $U \subseteq V$ is a subspace that satisfies $f(U) \subseteq U$, then there exists a complementary subspace $U^c$ such that $f(U^c) \subseteq U^c$.

\begin{lemma}
\label{lem:ss}
Let $f:V \to V$ be a semi-simple operator. The following hold:
	\begin{enumerate}
	\item\label{lem:ssPart1} For any $f$-invariant subspace $V' \subseteq V$, the restriction $f|_{V'}$ is semi-simple.
	\item\label{lem:ssPart4} For any $k \geq 1$, $f^k$ is semi-simple.
	\end{enumerate}
\end{lemma}

\begin{proof} 
The proof of Part \eqref{lem:ssPart1} is a straightforward application of Part \eqref{lem:DirectSumPart1} of Lemma \ref{lem:DirectSum}, so we omit the proof.

To prove Part \eqref{lem:ssPart4}, fix a semi-simple operator $f:V \to V$. 
The claim is that, for all $f^k$-invariant subspaces $U \subseteq V$, there exists an $f^k$-invariant subspace $W$ such that $U \oplus W = V$. We prove this by induction over $\dim V$ then by sub-induction over $\dim U$. The base cases where $\dim V = 0$ or $\dim U = 0$ are trivial, so we proceed to the induction step.

{\bf Claim 1}: We may assume $U$ is not contained in an intermediate $f$-invariant subspace.

If it were, say $U \subseteq W \subseteq V$, then $W$ has an $f$-invariant complement $W^c$ in $V$. By Lemma \ref{lem:ss}, Part \eqref{lem:ssPart1}, $f|_W$ is semi-simple, so by induction $U$ has an $f^p$-invariant complement $U^c$ in $W$. By Lemma \ref{lem:DirectSum}, Part \eqref{lem:DirectSumPart3}, $U^c + W^c$ is a complement of $U$ in $V$. This complement is $f^p$-invariant, so the lemma follows in this case.

{\bf Claim 2}: We may assume $U$ does not contain an intermediate $f^p$-invariant subspace.

If it did, say $W \subseteq U$, then by the subinduction hypothesis there exists an $f^p$-invariant complement $W^c$ of $W$ in $V$. By Lemma \ref{lem:DirectSum}, Part \eqref{lem:DirectSumPart1}, 
	\[U = W \oplus (U \cap W^c).\]
Notice that $U \cap W^c$ is also $f^p$-invariant and has dimension strictly less than $U$. 
Therefore the subinduction hypothesis again implies that it has an $f^p$-invariant complement $(U \cap W^c)^c$ in $V$. 
Since $U_1 = U \cap W^c$ is contained in the complement $W^c$ of $U_2 = W$, Part \eqref{lem:DirectSumPart2} of Lemma \ref{lem:DirectSum} implies that
	\[V = U \oplus \of{(U \cap W^c)^c \cap W^c}.\]
We have constructed a complement of $U$, and it is clear that it is the desired $f^p$-invariant complement.

\medskip
With the proofs of Claims 1 and 2 complete, we finish the proof of the theorem. Consider the subspace
	\[U + f(U) + f^2(U) + \ldots + f^{k-1}(U).\]
This is $f$-invariant since $U$ is $f^k$-invariant. This subspace contains $U$, so by Claim 1 we may assume it equals $U$ or $V$. In the former case, $U$ is already $f$-invariant and hence has an $f$-invariant complement $U^c$, and hence the proof is complete since $U^c$ is also $f^k$-invariant. Therefore we may assume this subspace is $V$.

Now consider more generally the $f^k$-invariant subspaces
	\[W_i = f^i(U) + f^{i+1}(U) + \ldots + f^{k-1}(U)\]
where $0 \leq i \leq k-1$. Assume $i$ is maximal such that $W_i$ contains $U$. By the argument in the previous paragraph, we may assume that $W_i = V$. We claim that $W_{i+1}$ is a complement to $U$ and hence a solution to our problem. To see this, we first note that $U \cap W_{i+1} = 0$ since it is an $f^k$-invariant subspace of $U$ that is neither $U$ by maximality of $i$ nor non-zero and strictly contained in $U$ by Claim 2. To complete the proof, it suffices to prove that $U$ and $W_{i+1}$ span $V$, and we prove this by counting dimensions. Note that $f^i(U) + W_{i+1} = W_i = V$, so
	\[\dim f^i(U) + \dim W_{i+1} \geq \dim V.\] 
Since $\dim U \geq \dim f^i(U)$, we have $\dim U + \dim W_{i+1} = \dim V$, as needed.
\end{proof}

\bibliographystyle{alpha}
\bibliography{bibfile}

\end{document}

%% file: preamble.tex
\usepackage[margin=1.25in]{geometry}

	\usepackage{comment}

	\usepackage[utf8]{inputenc}

	\usepackage{amsmath,amsfonts,amssymb,amsthm}

	\usepackage[T1]{fontenc}

	\usepackage{etoolbox}
	\patchcmd{\section}{\scshape}{\scshape\bfseries}{}{}
	\makeatletter
	\renewcommand{\@secnumfont}{\scshape\bfseries}
	\makeatother

	\let\epsilon\varepsilon

	\usepackage{parskip}

	\usepackage[bookmarks=true,bookmarksopen=true]{hyperref}

	\newtheorem{theorem}{Theorem}

	\newtheorem{corollary}[theorem]{Corollary}
        \newtheorem*{nonumbercorollary}{Corollary}
	\newtheorem{lemma}[theorem]{Lemma}
	\newtheorem{proposition}[theorem]{Proposition}
	\newtheorem{main}{Theorem}

	\numberwithin{theorem}{section}
	
	\theoremstyle{plain}

	\theoremstyle{definition}
	\newtheorem{definition}[theorem]{Definition}
	\newtheorem*{definition*}{Definition}
	\newtheorem{example}[theorem]{Example}
	
	\theoremstyle{remark}
	\newtheorem{remark}[theorem]{Remark}
	
	\numberwithin{equation}{section}
	
	\numberwithin{table}{section}

	\usepackage{tikz}
	\usetikzlibrary{matrix}

    \usepackage[all,cmtip]{xy}

	\usepackage{todonotes}

	\usepackage{enumerate}

	\usepackage{float}

	\makeatletter
	\def\blfootnote{\gdef\@thefnmark{}\@footnotetext}
	\makeatother

	\DeclareMathOperator{\codim}{codim}		\DeclareMathOperator{\cod}{\codim}

	\DeclareMathOperator{\Ric}{Ric}

	\newcommand{\of}[1]{\left( #1 \right)}
	\newcommand{\ceil}[1]{\left\lceil #1 \right\rceil}
	\newcommand{\floor}[1]{\left\lfloor #1 \right\rfloor}

	\newcommand{\gH}{\mathsf{H}}

	\newcommand{\gS}{\mathsf{S}}
	
	\newcommand{\SO}{\mathsf{SO}}

	\newcommand{\gT}{\mathsf{T}}
	\newcommand{\gU}{\mathsf{U}}
	
	\newcommand{\s}{\mathbb{S}}
	
	\newcommand{\CP}{\mathbb{C}\mathrm{P}}
	
	\newcommand{\HP}{\mathbb{H}\mathrm{P}}

	\newcommand{\bQ}{\mathbb{Q}}\newcommand{\Q}{\mathbb{Q}}

	\newcommand{\bZ}{\mathbb{Z}}	\newcommand{\Z}{\bZ}
							\newcommand{\Sq}{\operatorname{Sq}}
							\renewcommand{\P}{\operatorname{P}}
	
	\mathcode`l="8000
	\begingroup
	\makeatletter
	\lccode`\~=`\l
	\DeclareMathSymbol{\lsb@l}{\mathalpha}{letters}{`l}
	\lowercase{\gdef~{\ifnum\the\mathgroup=\m@ne \ell \else \lsb@l \fi}}%
	\endgroup

	\DeclareMathOperator{\End}{End}
	\DeclareMathOperator{\im}{im}

%% file: main.bbl
\begin{thebibliography}{KWW25}

\bibitem[Ada60]{Adams60}
J.F. Adams.
\newblock On the non-existence of elements of {H}opf invariant one.
\newblock {\em Ann. of Math. (2)}, 72:20--104, 1960.

\bibitem[Ade52]{Adem52}
J.~Adem.
\newblock {The iteration of the Steenrod squares in algebraic topology}.
\newblock {\em Proc. Nat. Acad. Sci. USA}, 38:720--726, 1952.

\bibitem[AP93]{AlldayPuppe93}
C.~Allday and V.~Puppe.
\newblock {\em Cohomological methods in transformation groups}, volume~32 of
  {\em Cambridge Studies in Advanced Mathematics}.
\newblock Cambridge University Press, Cambridge, 1993.

\bibitem[AQZ22]{AmannQuastZarei22preprint-v2}
Manuel Amann, Peter Quast, and Masoumeh Zarei.
\newblock The flavour of intermediate {Ricci} and homotopy when studying
  submanifolds of symmetric spaces.
\newblock preprint, arXiv:2010.15742, 2022.

\bibitem[AZ16]{AmannZiller16}
Manuel Amann and Wolfgang Ziller.
\newblock Geometrically formal homogeneous metrics of positive curvature.
\newblock {\em J. Geom. Anal.}, 26(2):996--1010, 2016.

\bibitem[Bur19]{Burdick19}
Bradley~Lewis Burdick.
\newblock Ricci-positive metrics on connected sums of projective spaces.
\newblock {\em Differential Geom. Appl.}, 62:212--233, 2019.

\bibitem[Bur20]{Burdick20}
Bradley~Lewis Burdick.
\newblock Metrics of positive {Ricci} curvature on the connected sums of
  products with arbitrarily many spheres.
\newblock {\em Ann. Global Anal. Geom.}, 58(4):433--476, 2020.

\bibitem[DDGR]{DVDVGARVpreprint}
Jason {DeVito}, Miguel {Dom\'inguez-V\'azquez}, David {Gonz\'alez-\'Alvaro},
  and Alberto {Rodr\'iguez-V\'azquez}.
\newblock Positive curvature on products of spheres and their quotients via
  intermediate fatness.
\newblock preprint.

\bibitem[DGM22]{DVGAM22}
Miguel {Dom\'inguez-V\'azquez}, David {Gonz\'alez-\'Alvaro}, and Lawrence
  Mouill\'e.
\newblock Infinite families of manifolds of positive $k^{\rm th}$-intermediate
  {Ricci} curvature with $k$ small.
\newblock {\em Math. Ann.}, 2022.

\bibitem[Esc92]{Eschenburg92Manuscripta}
J.-H. Eschenburg.
\newblock Cohomology of biquotients.
\newblock {\em Manuscripta Math.}, 75(2):151--166, 1992.

\bibitem[GH82]{GroveHalperin82}
Karsten Grove and Stephen Halperin.
\newblock Contributions of rational homotopy theory to global problems in
  geometry.
\newblock {\em Publ. Math. Inst. Hautes Etudes Sci.}, 56(1):171--177, 1982.

\bibitem[GKZ20]{GoertschesKonstantisZoller20GKM}
Oliver Goertsches, Panagiotis Konstantis, and Leopold Zoller.
\newblock G{KM} theory and {H}amiltonian non-{K}\"ahler actions in dimension 6.
\newblock {\em Adv. Math.}, 368:107141, 17, 2020.

\bibitem[Gro02]{Grove02}
Karsten Grove.
\newblock Geometry of, and via, symmetries.
\newblock In {\em Conformal, Riemannian and Lagrangian Geometry}, pages 31--53.
  American Mathematical Society, 2002.

\bibitem[GS94]{GroveSearle94}
Karsten Grove and Catherine Searle.
\newblock Positively curved manifolds with maximal symmetry-rank.
\newblock {\em J. Pure Appl. Algebra}, 91:137--142, 1994.

\bibitem[GW20]{GuijarroWilhelm20}
Luis Guijarro and Frederick Wilhelm.
\newblock Restrictions on submanifolds via focal radius bounds.
\newblock {\em Math. Res. Lett.}, 27(1):115--139, 2020.

\bibitem[GW22]{GuijarroWilhelm22}
Luis Guijarro and Frederick Wilhelm.
\newblock A softer connectivity principle.
\newblock {\em Communications in Analysis and Geometry}, 30(5):1093--1119,
  2022.

\bibitem[HK]{HarperKennard-pre}
J.R. Harper and L.~Kennard.
\newblock Extending {A}dams' theorem from singly generated to periodic
  cohomology.
\newblock preprint, arXiv:2412.16340.

\bibitem[Ken13]{Kennard13}
Lee Kennard.
\newblock On the {Hopf} conjecture with symmetry.
\newblock {\em Geom. Topol.}, 17(1):563--593, 2013.

\bibitem[KM24]{KennardMouille24}
Lee Kennard and Lawrence Mouillé.
\newblock Positive intermediate {Ricci} curvature with maximal symmetry rank.
\newblock {\em J. Geom. Anal.}, 34(5), 2024.

\bibitem[KWW21]{KennardWiemelerWilking21preprint}
Lee Kennard, Michael Wiemeler, and Burkhard Wilking.
\newblock Splitting of torus representations and applications in the {Grove}
  symmetry program.
\newblock preprint, arXiv:2106.14723v1, 2021.

\bibitem[KWW25]{KennardWiemelerWilking25}
Lee Kennard, Michael Wiemeler, and Burkhard Wilking.
\newblock Positive curvature, torus symmetry, and matroids.
\newblock {\em J. Eur. Math. Soc.}, 2025.
\newblock published online first.

\bibitem[Mou22]{Mouille22b}
Lawrence Mouill\'e.
\newblock Torus actions on manifolds with positive intermediate {Ricci}
  curvature.
\newblock {\em J. Lond. Math. Soc.}, 106(4):3792--3821, 2022.

\bibitem[MW13]{MareWillems13}
Augustin-Liviu Mare and Matthieu Willems.
\newblock Topology of the octonionic flag manifold.
\newblock {\em M\"unster J. Math.}, 6(2):483--523, 2013.

\bibitem[Nie]{Nienhaus-pre}
Jan Nienhaus.
\newblock An improved four-periodicity theorem and a conjecture of {H}opf with
  symmetry.
\newblock preprint, arXiv:2211.13151.

\bibitem[Nie18]{NienhausMSc}
Jan Nienhaus.
\newblock Spaces with partially periodic cohomology rings.
\newblock Master's thesis, Westf\"alische Wilhelms-Universit\"at M\"unster,
  2018.

\bibitem[Per97]{Perelman97}
Grigori Perelman.
\newblock Construction of manifolds of positive {Ricci} curvature with big
  volume and large {Betti} numbers.
\newblock In {\em Comparison Geometry}, volume~30, pages 157--163, 1997.

\bibitem[RW23a]{ReiserWraith23}
Philipp Reiser and David Wraith.
\newblock Intermediate {Ricci} curvatures and {Gromov’s Betti} number bound.
\newblock {\em J. Geom. Anal.}, 33(364), 2023.

\bibitem[RW23b]{ReiserWraith23preprint}
Philipp Reiser and David~J. Wraith.
\newblock A generalization of the {Perelman} gluing theorem and applications.
\newblock preprint, arXiv:2308.06996, 2023.

\bibitem[She90]{Shen90}
Zhongmin Shen.
\newblock {\em Finite topological type and vanishing theorems for {Riemannian}
  manifolds}.
\newblock PhD thesis, State University of New York at Stony Brook, 1990.

\bibitem[She93]{Shen93}
Zhongmin Shen.
\newblock On complete manifolds of nonnegative kth-{Ricci} curvature.
\newblock {\em Trans. Amer. Math. Soc.}, 338(1):289--310, 1993.

\bibitem[SY91]{ShaYang91}
Ji-Ping Sha and DaGang Yang.
\newblock Positive {Ricci} curvature on the connected sums of ${S}^n\times
  {S}^m$.
\newblock {\em J. Differential Geom.}, 33(1):127--137, 1991.

\bibitem[Wil97]{Wilhelm97}
Frederick Wilhelm.
\newblock On intermediate {Ricci} curvature and fundamental groups.
\newblock {\em Illinois J. Math.}, 41(3):488--494, 1997.

\bibitem[Wil03]{Wilking03}
Burkhard Wilking.
\newblock Torus actions on manifolds of positive sectional curvature.
\newblock {\em Acta Math.}, 191(2):259--297, 2003.

\bibitem[Xia97]{Xia97}
Changyu Xia.
\newblock A generalization of the classical sphere theorem.
\newblock {\em Proc. Amer. Math. Soc.}, 125(1):255--258, 1997.

\end{thebibliography}
